\documentclass[a4paper,11pt]{article}
\usepackage{amsmath}
\usepackage{amssymb}
\usepackage{theorem}
\usepackage{pstricks}
\usepackage{url}
\usepackage{euscript}
\usepackage{epic,eepic}
\usepackage{graphicx}
\PassOptionsToPackage{normalem}{ulem}
\renewcommand{\leq}{\ensuremath{\leqslant}}
\renewcommand{\geq}{\ensuremath{\geqslant}}
\newcommand{\minimize}[2]{\ensuremath{\underset{\substack{{#1}}}%
{\mathrm{minimize}}\;\;#2 }}

\newcommand{\menge}[2]{\big\{{#1} \mid {#2}\big\}} 

\newcommand{\Menge}[2]{\bigg\{{#1}~\bigg|~{#2}\bigg\}}

\newcommand{\emp}{\ensuremath{{\varnothing}}}

\newcommand{\sscal}[2]{\langle{#1}\mid {#2}\rangle} 
\newcommand{\scal}[2]{\left\langle{#1}\mid {#2} \right\rangle} 
\newcommand{\pscal}[2]{\langle\langle{#1}\mid{#2}\rangle\rangle} 
\newcommand{\psscal}[2]{\langle\langle\langle{#1}\mid{#2}%
\rangle\rangle\rangle} 

\newcommand{\infconv}{\ensuremath{\mbox{\footnotesize$\,\square\,$}}}
\newcommand{\exi}{\ensuremath{\exists\,}}
\newcommand{\zeroun}{\ensuremath{\left]0,1\right[}}   
\newcommand{\HH}{\ensuremath{\mathcal H}}
\newcommand{\GG}{\ensuremath{\mathcal G}}
\newcommand{\BL}{\ensuremath{\EuScript B}\,}
\newcommand{\SL}{\ensuremath{\EuScript S}\,}
\newcommand{\BP}{\ensuremath{\EuScript P}}

\newcommand{\KKK}{\ensuremath{\boldsymbol{\mathcal K}}}
\newcommand{\GGG}{\ensuremath{\boldsymbol{\mathcal G}}}

\newcommand{\sri}{\ensuremath{\operatorname{sri}}}
\newcommand{\reli}{\ensuremath{\operatorname{ri}\,}}
\newcommand{\RR}{\ensuremath{\mathbb R}}

\newcommand{\RP}{\ensuremath{\left[0,+\infty\right[}}
\newcommand{\RPP}{\ensuremath{\,\left]0,+\infty\right[}}
\newcommand{\RX}{\ensuremath{\,\left]-\infty,+\infty\right]}}
\newcommand{\RM}{\ensuremath{\,\left]-\infty,0\right]}}
\newcommand{\RXX}{\ensuremath{\left[-\infty,+\infty\right]}}
\newcommand{\NN}{\ensuremath{\mathbb N}}
\newcommand{\dom}{\ensuremath{\operatorname{dom}}}

\newcommand{\prox}{\ensuremath{\operatorname{prox}}}

\newcommand{\inte}{\ensuremath{\operatorname{int}}}
\newcommand{\cart}{\ensuremath{\mbox{\huge{$\times$}}}}

\newcommand{\RPX}{\ensuremath{{[0,\pinf]}}}

\newcommand{\Argmin}{\ensuremath{\operatorname{Argmin}}}
\newcommand{\argmin}{\ensuremath{\operatorname{argmin}}}
\newcommand{\ran}{\ensuremath{\operatorname{ran}}}
\newcommand{\zer}{\ensuremath{\operatorname{zer}}}
\newcommand{\gra}{\ensuremath{\operatorname{gra}}}

\newcommand{\cone}{\ensuremath{\operatorname{cone}}}
\newcommand{\spc}{\ensuremath{\overline{\operatorname{span}}\,}}

\newcommand{\Id}{\ensuremath{\operatorname{Id}}}

\newcommand{\weakly}{\ensuremath{\rightharpoonup}}
\newcommand{\minf}{\ensuremath{-\infty}}
\newcommand{\pinf}{\ensuremath{+\infty}}
\newtheorem{theorem}{Theorem}[section]
\newtheorem{lemma}[theorem]{Lemma}

\newtheorem{corollary}[theorem]{Corollary}
\newtheorem{proposition}[theorem]{Proposition}
\newtheorem{definition}[theorem]{Definition}
\theoremstyle{plain}{\theorembodyfont{\rmfamily}
}
\theoremstyle{plain}{\theorembodyfont{\rmfamily}
}
\theoremstyle{plain}{\theorembodyfont{\rmfamily}
}
\theoremstyle{plain}{\theorembodyfont{\rmfamily}
\newtheorem{example}[theorem]{Example}}
\theoremstyle{plain}{\theorembodyfont{\rmfamily}
\newtheorem{problem}[theorem]{Problem}}
\theoremstyle{plain}{\theorembodyfont{\rmfamily}
\newtheorem{remark}[theorem]{Remark}}
\theoremstyle{plain}{\theorembodyfont{\rmfamily}
\newtheorem{notation}[theorem]{Notation}}

\numberwithin{equation}{section}

\begin{document}
\title{\sffamily Variable Metric Forward-Backward Splitting 
with Applications to Monotone Inclusions in 
Duality\thanks{Contact author: 
P. L. Combettes, {\ttfamily{plc@math.jussieu.fr}},
phone:+33 1 4427 6319, fax:+33 1 4427 7200.
The work of B$\grave{\text{\u{a}}}$ng C\^ong V\~u was supported 
by the Vietnam National Foundation for Science and Technology 
Development.}}
\author{Patrick L. Combettes and 
B$\grave{\text{\u{a}}}$ng C\^ong V\~u\\[4mm]
\small UPMC Universit\'e Paris 06\\
\small Laboratoire Jacques-Louis Lions-- UMR CNRS 7598\\
\small 75005 Paris, France\\
\small{\ttfamily{plc@math.jussieu.fr}}, 
{\ttfamily{vu@ljll.math.upmc.fr}} 
}
\bigskip

\date{~}
\maketitle

\begin{abstract}
We propose a variable metric forward-backward splitting algorithm
and prove its convergence in real Hilbert spaces. We then use this
framework to derive primal-dual splitting algorithms for solving 
various classes of monotone inclusions in duality. Some of these
algorithms are new even when specialized to the fixed metric case. 
Various applications are discussed.
\end{abstract}



\section{Introduction} 

The forward-backward algorithm has a long history going back to 
the projected gradient method (see \cite{Sico10,Opti04} for
historical background). It addresses the problem of finding a zero
of the sum of two operators acting on a real Hilbert space $\HH$, 
namely, 
\begin{equation}
\label{e:A+B}
\text{find}\;\;x\in\HH\quad\text{such that}\quad 
0\in Ax+Bx,
\end{equation}
under the assumption that $A\colon\HH\to 2^{\HH}$ is maximally 
monotone and that $B\colon\HH\to\HH$ is $\beta$-cocoercive for some 
$\beta\in\RPP$, i.e. \cite{Livre1},
\begin{equation}
\label{e:cocoercive}
(\forall x\in\HH)(\forall y\in\HH)\quad
\scal{x-y}{Bx-By}\geq\beta\|Bx-By\|^2.
\end{equation}
This framework is quite central due to the large class of problems 
it encompasses in areas such as partial differential equations, 
mechanics, evolution inclusions, signal and image processing, 
best approximation, convex optimization, learning theory, inverse 
problems, statistics, game theory, and variational inequalities 
\cite{Sico10,Livre1,Bric13,Chen97,Opti04,Banf11,Smms05,Devi11,%
Duch09,Facc03,Glow89,Merc79,Merc80,Tsen90,Tsen91,Zhud96}. 
The forward-backward algorithm operates according to the routine 
\begin{equation}
\label{e:2012-05-13}
x_0\in\HH\quad\text{and}\quad
(\forall n\in\NN)\quad x_{n+1}=(\Id+\gamma_nA)^{-1}
(x_n-\gamma_nBx_n),\quad\text{where}\quad 0<\gamma_n<2\beta.
\end{equation}
In classical optimization methods, the benefits of changing the 
underlying metric over the course of the iterations to improve 
convergence profiles has long been recognized \cite{Davi59,Pear69}.
In proximal methods, variable metrics have been investigated
mostly when $B=0$ in \eqref{e:A+B}. In such instances 
\eqref{e:2012-05-13} reduces to the proximal point algorithm 
\begin{equation}
\label{e:2012-05-14}
x_0\in\HH\quad\text{and}\quad
(\forall n\in\NN)\quad x_{n+1}=(\Id+\gamma_nA)^{-1}x_n,
\quad\text{where}\quad
\gamma_n>0.
\end{equation}
In the case when $A$ is the subdifferential of a real-valued
convex function in a finite dimensional setting, variable 
metric versions of \eqref{e:2012-05-14} have been proposed in
\cite{Bonn95,Chen99,Lema97,Qi95}. These methods draw heavily 
on the fact that the proximal point algorithm for minimizing
a function corresponds to the gradient descent method applied 
to its Moreau envelope.
In the same spirit, variable metric proximal point algorithms 
for a general maximally monotone operator $A$ were considered in 
\cite{Burk99,Qian92}. In \cite{Burk99}, superlinear 
convergence rates were shown to be achievable under suitable 
hypotheses (see also \cite{Burk00} for further developments).
The finite dimensional variable metric proximal point algorithm 
proposed in \cite{Pare08} allows for errors in the proximal steps
and features a flexible class of exogenous metrics to implement 
the algorithm. The first variable metric forward-backward algorithm
appears to be that introduced in \cite[Section~5]{Chen97}. It 
focuses on linear convergence results in the case when 
$A+B$ is strongly monotone and $\HH$ is finite-dimensional.
The variable metric splitting algorithm of \cite{Loli09} provides a 
framework which can be used to solve \eqref{e:A+B} in instances 
when $\HH$ is finite-dimensional and $B$ is merely Lipschitzian. 
However, it does not exploit the cocoercivity property 
\eqref{e:cocoercive} and it is more cumbersome to implement 
than the forward-backward iteration. Let us 
add that, in the important case when $B$ is the gradient of a 
convex function, the Baillon-Haddad theorem asserts that the 
notions of cocoercivity and Lipschitz-continuity coincide 
\cite[Corollary~18.16]{Livre1}.

The goal of this paper is two-fold. First, we propose a general
purpose variable metric forward-backward algorithm to solve
\eqref{e:A+B}--\eqref{e:cocoercive} in Hilbert spaces and analyze
its asymptotic behavior, both in terms of weak and strong
convergence. Second, we show that this algorithm can be used to 
solve a broad class of composite monotone inclusion problems
in duality by formulating them as instances of 
\eqref{e:A+B}--\eqref{e:cocoercive} in alternate Hilbert spaces. 
Even when restricted to the constant metric case, some of these
results are new.

The paper is organized as follows. Section~\ref{sec:2} is 
devoted to notation and background. In Section~\ref{sec:3},
we provide preliminary results. The variable metric forward-backward 
algorithm is introduced and analyzed in Section~\ref{sec:4}.
In Section~\ref{sec:5}, we present a new variable metric 
primal-dual splitting algorithm for strongly monotone composite 
inclusions. This algorithm is obtained by applying the
forward-backward algorithm of Section~\ref{sec:4} to the dual
inclusion. In Section~\ref{sec:6}, we consider a more
general class of composite inclusions in duality and show that
they can be solved by applying the
forward-backward algorithm of Section~\ref{sec:4} to 
a certain inclusion problem posed in the primal-dual product space.
Applications to minimization problems, variational inequalities,
and best approximation are discussed. 

\section{Notation and background}
\label{sec:2}

We recall some notation and background from convex analysis and 
monotone operator theory (see \cite{Livre1} for a detailed account).

Throughout, $\HH$, $\GG$, and $(\GG_i)_{1\leq i\leq m}$ are real 
Hilbert spaces. We denote the scalar product of a Hilbert space by 
$\scal{\cdot}{\cdot}$ and the associated norm by $\|\cdot\|$. 
The symbols $\weakly$ and $\to$ denote respectively weak and strong 
convergence, and $\Id$ denotes the identity operator.
We denote by $\BL(\HH,\GG)$ the space of bounded linear operators 
from $\HH$ to $\GG$, we set $\BL(\HH)=\BL(\HH,\HH)$ and
$\SL(\HH)=\menge{L\in\BL(\HH)}{L=L^*}$, where $L^*$ denotes the
adjoint of $L$. The Loewner partial ordering on $\SL(\HH)$ is 
defined by
\begin{equation}
\label{e:loewner}
(\forall U\in\SL(\HH))(\forall V\in\SL(\HH))\quad
U\succcurlyeq V\quad\Leftrightarrow\quad(\forall x\in\HH)\quad
\scal{Ux}{x}\geq\scal{Vx}{x}.
\end{equation}
Now let $\alpha\in\RP$. We set
\begin{equation}
\BP_{\alpha}(\HH)=\menge{U\in\SL(\HH)}{U\succcurlyeq\alpha\Id},
\end{equation}
and we denote by $\sqrt{U}$ the square root of 
$U\in\BP_{\alpha}(\HH)$. Moreover, for every 
$U\in\BP_\alpha(\HH)$, we define a semi-scalar product and 
a semi-norm (a scalar product and a norm if $\alpha>0$) by
\begin{equation}
\label{Unorm}
(\forall x\in\HH)(\forall y\in\HH)\quad\scal{x}{y}_U=\scal{Ux}{y}
\quad\text{and}\quad\|x\|_U=\sqrt{\scal{Ux}{x}}.
\end{equation}

\begin{notation}
\label{n:palawan-mai2008-}
We denote by $\GGG=\GG_1\oplus\cdots\oplus\GG_m$ the Hilbert direct 
sum of the Hilbert spaces $(\GG_i)_{1\leq i\leq m}$, i.e., their 
product space equipped with the scalar product and the associated 
norm respectively defined by
\begin{equation}
\label{e:palawan-mai2008-}
\pscal{\cdot}{\cdot}
\colon(\boldsymbol{x},\boldsymbol{y})\mapsto
\sum_{i=1}^m\scal{x_i}{y_i}
\quad\text{and}\quad|||\cdot|||\colon
\boldsymbol{x}\mapsto\sqrt{\sum_{i=1}^m\|x_i\|^2},
\end{equation}
where ${\boldsymbol x}=(x_i)_{1\leq i\leq m}$ and 
${\boldsymbol y}=(y_i)_{1\leq i\leq m}$ denote generic elements in
$\GGG$.
\end{notation}

Let $A\colon\HH\to 2^{\HH}$ be a set-valued operator.
The domain and the graph of $A$ are respectively defined by 
$\dom A=\menge{x\in\HH}{Ax\neq\emp}$ and 
$\gra A=\menge{(x,u)\in\HH\times\HH}{u\in Ax}$.
We denote by $\zer A=\menge{x\in\HH}{0\in Ax}$ the set of zeros 
of $A$ and by 
$\ran A=\menge{u\in\HH}{(\exists\; x\in\HH)\;u\in Ax}$ 
the range of $A$. The inverse of $A$ is
$A^{-1}\colon\HH\mapsto 2^{\HH}\colon u\mapsto 
\menge{x\in\HH}{u\in Ax}$, and the resolvent of $A$ is
\begin{equation}
\label{e:resolvent}
J_A=(\Id+A)^{-1}.
\end{equation}
Moreover, $A$ is monotone if 
\begin{equation}
(\forall(x,y)\in\HH\times\HH)
(\forall(u,v)\in Ax\times Ay)\quad\scal{x-y}{u-v}\geq 0,
\end{equation}
and maximally monotone if it is monotone and there exists no 
monotone operator 
$B\colon\HH\to2^\HH$ such that $\gra A\subset\gra B$ and $A\neq B$.
The parallel sum of $A$ and $B\colon\HH\to 2^{\HH}$ is 
\begin{equation}
\label{e:parasum}
A\infconv B=(A^{-1}+ B^{-1})^{-1}.
\end{equation} 
The conjugate of $f\colon\HH\to\RX$ is 
\begin{equation}
\label{e:conjugate}
f^*\colon\HH\to\RXX\colon 
u\mapsto \sup_{x\in\HH}\big(\scal{x}{u}-f(x)\big), 
\end{equation}
and the infimal convolution of $f$ with $g\colon\HH\to\RX$ is
\begin{equation}
f\infconv g\colon\HH\to\RXX\colon x\mapsto
\inf_{y\in\HH}\big(f(y)+g(x-y)\big).
\end{equation}
The class of lower semicontinuous convex functions $f\colon\HH\to\RX$
such that $\dom f=\menge{x\in\HH}{f(x)<\pinf}\neq\emp$ is denoted by 
$\Gamma_0(\HH)$. If $f\in\Gamma_0(\HH)$, then $f^*\in\Gamma_0(\HH)$
and the subdifferential of $f$ is the maximally monotone operator
\begin{equation}
\partial f\colon\HH\to 2^{\HH}\colon x
\mapsto\menge{u\in\HH}{(\forall y\in\HH)\;
\scal{y-x}{u}+f(x)\leq f(y)}
\end{equation} 
with inverse $(\partial f)^{-1}=\partial f^*$.
Let $C$ be a nonempty subset of $\HH$. 
The indicator function and the distance
function of $C$ are defined on $\HH$ as 
\begin{equation}
\iota_C\colon x\mapsto
\begin{cases}
0,&\text{if}\;\;x\in C;\\
\pinf,&\text{if}\;\;x\notin C
\end{cases}
\qquad\text{and}\quad
d_C=\iota_C\infconv\|\cdot\|\colon x\mapsto\inf_{y\in C}\|x-y\|,
\end{equation}
respectively, the interior of $C$ is $\inte C$,
and the support function of $C$ is 
$\sigma_C=\iota_C^*$. Now suppose that $C$ is convex. The normal 
cone operator of $C$ is
\begin{equation}
\label{e:normalcone}
N_C=\partial\iota_C\colon\HH\to 2^{\HH}\colon x\mapsto
\begin{cases}
\menge{u\in\HH}{(\forall y\in C)\;\;\scal{y-x}{u}\leq 0},
&\text{if}\;x\in C;\\
\emp,&\text{otherwise,}
\end{cases}
\end{equation}
and the strong relative interior of $C$, i.e., the set of points 
$x\in C$ such that the conical hull of $-x+C$ is a closed vector 
subspace of $\HH$, is denoted by $\sri C$; if $\HH$ is 
finite-dimensional, $\sri C$ coincides with the relative interior 
of $C$, denoted by $\reli C$. If $C$ is also closed, its projector 
is denoted by $P_C$, i.e.,
$P_C\colon\HH\to C\colon x\mapsto{\argmin}_{y\in C}\|x-y\|$.

Finally, $\ell_+^1(\NN)$ denotes the set of summable sequences in 
$\RP$.

\section{Preliminary results}
\label{sec:3}

\subsection{Technical results}

The following properties can be found in 
\cite[Section~VI.2.6]{Kato80} (see \cite[Lemma~2.1]{Guad2012} for 
an alternate short proof).

\begin{lemma}
\label{l:kjMMXII}
Let $\alpha\in\RPP$, let $\mu\in\RPP$, and let $A$ and $B$ be 
operators in $\SL(\HH)$ such that 
$\mu\Id\succcurlyeq A\succcurlyeq B\succcurlyeq\alpha\Id$. Then
the following hold.
\begin{enumerate}
\item
\label{l:kjMMXII-i}
$\alpha^{-1}\Id\succcurlyeq B^{-1}\succcurlyeq A^{-1}\succcurlyeq
\mu^{-1}\Id$. 
\item
\label{l:kjMMXII-ii}
$(\forall x\in\HH)$ $\scal{A^{-1}x}{x}\geq\|A\|^{-1}\|x\|^2$.
\item
\label{l:kjMMXII-iii}
$\|A^{-1}\|\leq\alpha^{-1}$.
\end{enumerate}
\end{lemma}

The next fact concerns sums of composite cocoercive operators.
\begin{proposition} 
\label{p:samurai11}
Let $I$ be a finite index set. For every $i\in I$, let 
$0\neq L_i\in\BL(\HH,\GG_i)$, let $\beta_i\in\RPP$, and let
$T_i\colon\GG_i\to\GG_i$ be $\beta_i$-cocoercive. Set 
$T=\sum_{i\in I}L_i^*T_iL_i$ and 
$\beta=1/\big(\sum_{i\in I}\|L_i\|^2/\beta_i\big)$.
Then $T$ is $\beta$-cocoercive.
\end{proposition}
\begin{proof}
Set $(\forall i\in I)$ $\alpha_i=\beta\|L_i\|^2/\beta_i$. Then 
$\sum_{i\in I}\alpha_i=1$ and, by convexity of $\|\cdot\|^2$ and 
\eqref{e:cocoercive}, 
\begin{align}
(\forall x\in\HH)(\forall y\in\HH)\quad
\scal{x-y}{Tx-Ty}
&=\sum_{i\in I}\scal{x-y}{L_i^*T_iL_ix-L_i^*T_iL_iy}\nonumber\\
&=\sum_{i\in I}\scal{L_ix-L_iy}{T_iL_ix-T_iL_iy}\nonumber\\
&\geq\sum_{i\in I}\beta_i\|T_iL_ix-T_iL_iy\|^2\nonumber\\
&\geq\sum_{i\in I}\frac{\beta_i}{\|L_i\|^2}
\|L_i^*T_iL_ix-L_i^*T_iL_iy\|^2\nonumber\\
&=\beta\sum_{i\in I}\alpha_i\Big\|\frac{1}{\alpha_i}
(L_i^*T_iL_ix-L_i^*T_iL_iy)\Big\|^2\nonumber\\
&\geq\beta\Big\|\sum_{i\in I}(L_i^*T_iL_ix-L_i^*T_iL_iy)\Big\|^2
\nonumber\\
&=\beta\|Tx-Ty\|^2,
\end{align}
which concludes the proof.
\end{proof}

\subsection{Variable metric quasi-Fej\'er sequences}

The following results are from \cite{Guad2012}.

\begin{proposition}
\label{p:1} 
Let $\alpha\in\RPP$, let $(W_n)_{n\in\NN}$ be in 
$\BP_{\alpha}(\HH)$, let $C$ be a nonempty subset of $\HH$, 
and let $(x_n)_{n\in\NN}$ be a sequence in $\HH$ such that 
\begin{multline}
\label{e:vmqf1}
\big(\exi(\eta_n)_{n\in\NN}\in\ell_+^1(\NN)\big)
\big(\forall z\in C\big)\big(\exi(\varepsilon_n)_{n\in\NN}\in
\ell_+^1(\NN)\big)(\forall n\in\NN)\\
\|x_{n+1}-z\|_{W_{n+1}}\leq(1+\eta_n)
\|x_n-z\|_{W_n}+\varepsilon_n.
\end{multline}
Then $(x_n)_{n\in\NN}$ is bounded and, for every $z\in C$,
$(\|x_n-z\|_{W_n})_{n\in\NN}$ converges.
\end{proposition}

\begin{proposition}
\label{p:guad1} 
Let $\alpha\in\RPP$, and let $(W_n)_{n\in\NN}$ and $W$ be 
operators in $\BP_{\alpha}(\HH)$ such that $W_n\to W$ pointwise, 
as is the case when 
\begin{equation}
\sup_{n\in\NN}\|W_n\|<\pinf\quad\text{and}\quad
(\exi (\eta_n)_{n\in\NN}\in\ell_+^1(\NN))
(\forall n\in\NN)\quad (1+\eta_n)W_n\succcurlyeq W_{n+1}.
\end{equation}
Let $C$ be a nonempty subset of 
$\HH$, and let $(x_n)_{n\in\NN}$ be a sequence in $\HH$ such 
that \eqref{e:vmqf1} is satisfied. Then 
$(x_n)_{n\in\NN}$ converges weakly to a point in $C$ 
if and only if every weak sequential cluster point of 
$(x_n)_{n\in\NN}$ is in $C$.
\end{proposition}

\begin{proposition}
\label{p:monodc} 
Let $\alpha\in\RPP$, let $(W_n)_{n\in\NN}$ be a sequence in 
$\BP_{\alpha}(\HH)$ such that $\sup_{n\in\NN}\|W_n\|<\pinf$, let 
$C$ be a nonempty closed subset of $\HH$, and let 
$(x_n)_{n\in\NN}$ be a sequence in $\HH$ such that 
\begin{multline}
\label{e:vmqf2}
\big(\exi(\varepsilon_n)_{n\in\NN}\in\ell_+^1(\NN)\big)
\big(\exi(\eta_n)_{n\in\NN}\in\ell_+^1(\NN)\big)(\forall z\in C)
(\forall n\in\NN)\\
\|x_{n+1}-z\|_{W_{n+1}}\leq(1+\eta_n)\|x_n-z\|_{W_n}+\varepsilon_n.
\end{multline}
Then $(x_n)_{n\in\NN}$ converges strongly to a point in $C$ if 
and only if $\varliminf d_C(x_n)=0$.
\end{proposition}

\begin{proposition}
\label{p:2012-03-26}
Let $\alpha\in\RPP$, let $(\nu_n)_{n\in\NN}\in\ell_+^1(\NN)$, and
let $(W_n)_{n\in\NN}$ be a sequence in $\BP_{\alpha}(\HH)$
such that $\sup_{n\in\NN}\|W_n\|<\pinf$ and
$(\forall n\in\NN)$ $(1+\nu_n)W_{n+1}\succcurlyeq W_n$.
Furthermore, let $C$ be a subset of $\HH$ such that 
$\inte C\neq\emp$, let $z\in C$ and $\rho\in\RPP$ be such that
$B(z;\rho)\subset C$, and let $(x_n)_{n\in\NN}$ be a sequence in 
$\HH$ such that 
\begin{multline}
\label{e:vmqf3}
\big(\exi(\varepsilon_n)_{n\in\NN}\in\ell_+^1(\NN)\big)
\big(\exi(\eta_n)_{n\in\NN}\in\ell_+^1(\NN)\big)
(\forall x\in B(z;\rho))(\forall n\in\NN)\\
\|x_{n+1}-x\|^2_{W_{n+1}}\leq
(1+\eta_n)\|x_n-x\|^2_{W_n}+\varepsilon_n.
\end{multline}
Then $(x_n)_{n\in\NN}$ converges strongly.
\end{proposition}

\subsection{Monotone operators}

We establish some results on monotone operators in a variable
metric environment.

\begin{lemma}
\label{l:maxmon45}
Let $A\colon\HH\to 2^{\HH}$ be maximally monotone, let
$\alpha\in\RPP$, let $U\in{\EuScript P}_\alpha(\HH)$, and let 
$\GG$ be the real Hilbert space obtained by endowing $\HH$ with 
the scalar product 
$(x,y)\mapsto\scal{x}{y}_{U^{-1}}=\scal{x}{U^{-1}y}$.
Then the following hold.
\begin{enumerate}
\item
\label{l:maxmon45i}
$UA\colon\GG\to 2^{\GG}$ is maximally monotone.
\item
\label{l:maxmon45ii}
$J_{UA}\colon\GG\to\GG$ is $1$-cocoercive, i.e., 
firmly nonexpansive, hence nonexpansive.
\item
\label{l:maxmon45iii}
$J_{UA}=(U^{-1}+A)^{-1}\circ U^{-1}$.
\end{enumerate}
\end{lemma}
\begin{proof}
\ref{l:maxmon45i}:
Set $B=UA$ and $V=U^{-1}$. For every $(x,u)\in\gra B$ and every
$(y,v)\in\gra B$, $Vu\in VBx=Ax$ and $Vv\in VBy=Ay$, so that 
\begin{equation}
\scal{x-y}{u-v}_V=\scal{x-y}{Vu-Vv}\geq 0
\end{equation}
by monotonicity of $A$ on $\HH$. This shows that $B$ is monotone 
on $\GG$. Now let $(y,v)\in\HH^{2}$ be such that
\begin{equation}
\label{e:conon}
(\forall(x,u)\in\gra B)\quad\scal{x-y}{u-v}_{V}\geq 0.
\end{equation}
Then, for every $(x,u)\in\gra A$, $(x,Uu)\in\gra B$ and
we derive from \eqref{e:conon} that 
\begin{equation}
\label{e:conon1}
\scal{x-y}{u-Vv}=\scal{x-y}{Uu-v}_{V}\geq 0.
\end{equation}
Since $A$ is maximally monotone on $\HH$, \eqref{e:conon1} 
gives $(y,Vv)\in\gra A$, which implies that $(y,v)\in\gra B$. 
Hence, $B$ is maximally monotone on $\GG$.

\ref{l:maxmon45ii}: This follows from \ref{l:maxmon45i} and
\cite[Corollary~23.8]{Livre1}.

\ref{l:maxmon45iii}: Let $x$ and $p$ be in $\GG$. Then 
$p=J_{UA}x$ $\Leftrightarrow$ $x\in p+UAp$
$\Leftrightarrow$ $U^{-1}x\in(U^{-1}+A)p$
$\Leftrightarrow$ $p=(U^{-1}+A)^{-1}(U^{-1}x)$. 
\end{proof}

\begin{remark}
\label{r:2012-04-15}
let $\alpha\in\RPP$, let $U\in{\EuScript P}_\alpha(\HH)$, 
set $f\colon\HH\to\RR\colon x\mapsto\scal{U^{-1}x}{x}/2$, and 
let $D\colon(x,y)\mapsto f(x)-f(y)-\scal{x-y}{\nabla f(y)}$
be the associated Bregman distance. Then 
Lemma~\ref{l:maxmon45}\ref{l:maxmon45iii} asserts that
$J_{UA}=(\nabla f +A)^{-1}\circ\nabla f$. In other words, 
$J_{UA}$ is the $D$-resolvent of $A$ introduced in 
\cite[Definition~3.7]{Sico03}.
\end{remark}

Let $U\in\BP_{\alpha}(\HH)$ for some $\alpha\in\RPP$. 
The proximity operator of $f\in\Gamma_0(\HH)$ relative to 
the metric induced by $U$ is \cite[Section~XV.4]{Hull93} 
\begin{equation}
\prox^{U}_f\colon\HH\to\HH\colon x
\mapsto\underset{y\in\HH}{\argmin}\:f(y)+\frac12\|x-y\|_{U}^2,
\end{equation}
and the projector onto a nonempty closed convex subset $C$ of 
$\HH$ relative to the norm $\|\cdot\|_{U}$ is denoted by $P_C^U$. 
We have
\begin{equation}
\label{e:prox3}
\prox^{U}_{f}=J_{U^{-1}\partial f}\quad\text{and}\quad
P_C^U=\prox^{U}_{\iota_C},
\end{equation}
and we write $\prox_f^{\Id}=\prox_f$.

In the case when $U=\Id$ in Lemma~\ref{l:maxmon45}, examples of
closed form expressions for $J_{UA}$ and basic resolvent calculus 
rules can be found in \cite{Livre1,Banf11,Smms05}. A few examples 
illustrating the case when $U\neq\Id$ are provided below.
The first result is an extension of the well-known resolvent
identity $J_A+J_{A^{-1}}=\Id$.

\begin{example}
\label{ex:resolvent}
Let $\alpha\in\RPP$, let $\gamma\in\RPP$, and let 
$U\in\BP_{\alpha}(\HH)$. Then the following hold.
\begin{enumerate}
\item
\label{ex:resolventi}
Let $A\colon\HH\to 2^{\HH}$ be maximally monotone. Then
\begin{equation}
J_{\gamma UA}=\sqrt{U}J_{\gamma\sqrt{U}A\sqrt{U}}\sqrt{U}^{-1}
=\Id-\gamma UJ_{\gamma^{-1}U^{-1}A^{-1}}(\gamma^{-1}U^{-1}).
\end{equation}
\item
\label{ex:resolventii}
Let $f\in\Gamma_0(\HH)$. Then
$\prox^U_{\gamma f}=\sqrt{U^{-1}}\prox_{\gamma f\circ\sqrt{U^{-1}}}
\sqrt{U}=\Id-\gamma U^{-1}\prox^{U^{-1}}_{\gamma^{-1}f^*}
(\gamma^{-1}U)$.
\item
\label{ex:resolventiii}
Let $C$ be a nonempty closed convex subset of $\HH$. Then
$\prox^U_{\gamma\sigma_C}=\sqrt{U^{-1}}\prox_{\gamma\sigma_C
\circ\sqrt{U^{-1}}}\sqrt{U}=\Id-\gamma U^{-1}P_C^{U^{-1}}
(\gamma^{-1}U)$.
\end{enumerate}
\end{example}
\begin{proof}
\ref{ex:resolventi}:
Let $x$ and $p$ be in $\HH$. Then 
\begin{eqnarray}
\label{e:resolvent24}
p=J_{\gamma UA}x 
&\Leftrightarrow&x-p\in\gamma UAp\nonumber\\
&\Leftrightarrow&\sqrt{U}^{-1}x-\sqrt{U}^{-1}p
\in\gamma\sqrt{U}A\sqrt{U}\sqrt{U}^{-1}p\nonumber\\
&\Leftrightarrow&\sqrt{U}^{-1}p=
J_{\gamma\sqrt{U}A\sqrt{U}}\big(\sqrt{U}^{-1}x\big)\nonumber\\
&\Leftrightarrow&p=
\sqrt{U}J_{\gamma\sqrt{U}A\sqrt{U}}\big(\sqrt{U}^{-1}x\big).
\end{eqnarray}
Furthermore, by \cite[Proposition~23.23(ii)]{Livre1},
$J_{\sqrt{U}(\gamma A)\sqrt{U}}=\Id-\sqrt{U}
\big(U+(\gamma A)^{-1}\big)^{-1}\sqrt{U}$.
Hence, \eqref{e:resolvent24} yields
\begin{equation}
\label{e:resolvent26}
J_{\gamma UA}=\Id-U\big(U+(\gamma A)^{-1}\big)^{-1}.
\end{equation}
However 
\begin{eqnarray}
p=\big(U+(\gamma A)^{-1}\big)^{-1}x 
&\Leftrightarrow&x\in Up+(\gamma A)^{-1}p\nonumber\\
&\Leftrightarrow&\gamma^{-1}p\in A(x-Up)\nonumber\\
&\Leftrightarrow&x-Up\in A^{-1}(\gamma^{-1}p)\nonumber\\
&\Leftrightarrow&\gamma^{-1}U^{-1}x\in\big(\Id+
\gamma^{-1}U^{-1}A^{-1}\big)(\gamma^{-1}p)\nonumber\\
&\Leftrightarrow&\gamma^{-1}p=J_{\gamma^{-1}U^{-1}A^{-1}}
(\gamma^{-1}U^{-1}x).
\end{eqnarray}
Hence, $(U+(\gamma A)^{-1})^{-1}=\gamma 
J_{\gamma^{-1}U^{-1}A^{-1}}(\gamma^{-1}U^{-1})$ and, using 
\eqref{e:resolvent26}, we obtain the rightmost identity in 
\ref{ex:resolventi}.

\ref{ex:resolventii}: Apply \ref{ex:resolventi} to $A=\partial f$,
and use \eqref{e:prox3}.

\ref{ex:resolventiii}: Apply \ref{ex:resolventii} to $f=\sigma_C$, 
and use \eqref{e:prox3}.
\end{proof}

\begin{example}
Define $\GGG$ as in Notation~\ref{n:palawan-mai2008-}, let 
$\alpha\in\RR$, and, for every $i\in\{1,\ldots,m\}$, let 
$A_i\colon\GG_i\to 2^{\GG_i}$ be maximally 
monotone and let $U_i\in\BP_\alpha(\GG_i)$. Set 
$\boldsymbol{A}\colon\GGG\to 2^{\GGG}\colon
(x_i)_{1\leq i\leq m}\mapsto\cart_{\!i=1}^{\!m}A_ix_i$ and 
$\boldsymbol{U}\colon\GGG\to\GGG\colon(x_i)_{1\leq i\leq m}\mapsto
(U_ix_i)_{1\leq i\leq m}$. Then ${\boldsymbol U}{\boldsymbol A}$ is 
maximally monotone and
\begin{equation}
\label{e:noumea2010-09-02a}
(\forall (x_i)_{1\leq i\leq m}\in\GGG)\quad
J_{\boldsymbol{U}\boldsymbol{A}}(x_i)_{1\leq i\leq m}=
(J_{U_iA_i}x_i)_{1\leq i\leq m}.
\end{equation}
\end{example}
\begin{proof}
This follows from Lemma~\ref{l:maxmon45}\ref{l:maxmon45i} and 
\cite[Proposition~23.16]{Livre1}.
\end{proof}

\begin{example}
Let $\alpha\in\RPP$, let $\xi\in\RR$, let $U\in\BP_{\alpha}(\HH)$, 
let $\phi\in\Gamma_0(\RR)$, suppose that $0\neq u\in\HH$, and set 
$H=\menge{x\in\HH}{\scal{x}{u}\leq \xi}$ and 
$g=\phi(\scal{\cdot}{u})$. Then $g\in\Gamma_0(\HH)$ and 
\begin{equation}
\label{e:oda1}
(\forall x\in\HH)\quad\prox_{g}^Ux=x+
\displaystyle{\frac{\prox_{\|\sqrt{U^{-1}}u\|^2\phi}\scal{x}{u}
-\scal{x}{u}}{\|\sqrt{U^{-1}}u\|^2}}U^{-1}u
\end{equation}
and
\begin{equation}
\label{e:oda2}
P_{H}^{U}x=
\begin{cases}
x,&\text{if}\;\;\scal{x}{u}\leq\xi;\\
x+\displaystyle{\frac{\xi-\scal{x}{u}}{\scal{u}{U^{-1}u}}U^{-1}u},
&\text{if}\;\;\scal{x}{u}>\xi.
\end{cases}
\end{equation}
\begin{proof} 
It follows from Example \ref{ex:resolvent}\ref{ex:resolventii} 
that
\begin{align}
\label{e:redoex}
(\forall x\in\HH)\quad \prox_{g}^Ux 
&=\sqrt{U^{-1}}\prox_{g\circ\sqrt{U^{-1}}}\sqrt{U}x.
\end{align}
Moreover, $g\circ\sqrt{U^{-1}}=\phi(\sscal{\cdot}{\sqrt{U^{-1}}u})$.
Hence, using \eqref{e:redoex} and \cite[Corollary~23.33]{Livre1}, 
we obtain 
\begin{align}
(\forall x\in\HH)\quad\prox_{g}^Ux
&=\sqrt{U^{-1}}\prox_{\phi(\scal{\cdot}{\sqrt{U^{-1}}u})}\sqrt{U}x
\nonumber\\
&=x+\displaystyle{\frac{\prox_{\|\sqrt{U^{-1}}u\|^2\phi}\scal{x}{u}
-\scal{x}{u}}{\|\sqrt{U^{-1}}u\|^2}}U^{-1}u.
\end{align}
Finally, upon setting $\phi=\iota_{\left]\minf,\xi\right]}$,
we obtain \eqref{e:oda2} from \eqref{e:oda1}.
\end{proof}
\end{example}

\begin{example} 
\label{ex:r879}
Let $\alpha\in\RPP$, let $\gamma\in\RR$, let $A\in\BP_0(\HH)$,
let $u\in\HH$, let $U\in\BP_{\alpha}(\HH)$, and set 
$\varphi\colon\HH\to\RR\colon x\mapsto\scal{Ax}{x}/2
+\scal{x}{u}+\gamma$. Then $\varphi\in\Gamma_0(\HH)$ and
\begin{equation}
\label{e:id2}
(\forall x\in\HH)\quad
\prox_{\varphi}^{U}x=(\Id+U^{-1}A)^{-1}(x-U^{-1}u). 
\end{equation}
\begin{proof}
Let $x\in\HH$. Then
$p=\prox_{\varphi}^{U}x$ 
$\Leftrightarrow$ $x-p=U^{-1}\nabla\varphi(p)$
$\Leftrightarrow$ $x-p=U^{-1}(Ap+u)$
$\Leftrightarrow$ $x-U^{-1}u=(\Id+U^{-1}A)p$
$\Leftrightarrow$ $p=(\Id+U^{-1}A)^{-1}(x-U^{-1}u)$.
\end{proof}
\end{example}

\begin{example}
Let $\alpha\in\RPP$ and let $U\in\BP_{\alpha}(\HH)$.
For every $i\in\{1,\ldots,m\}$, let $r_i\in\GG_i$, 
let $\omega_i\in\RPP$, and let $L_i\in\BL(\HH,\GG_i)$.
Set $\varphi\colon x\mapsto(1/2)\sum_{i=1}^m\omega_i\|L_ix-r_i\|^2$.
Then $\varphi\in\Gamma_0(\HH)$ and 
\begin{equation}
\label{e:id3}
(\forall x\in\HH)\quad\prox_{\varphi}^{U}x=
\bigg(\Id+U^{-1}\sum_{i=1}^{m}\omega_iL_{i}^*L_i\bigg)^{-1}
\bigg(x+U^{-1}\sum_{i=1}^{m}\omega_iL_{i}^*r_i\bigg). 
\end{equation}
\end{example}
\begin{proof}
We have $\varphi\colon x\mapsto\scal{Ax}{x}/2+\scal{x}{u}+\gamma$, 
where $A=\sum_{i=1}^{m}\omega_iL_{i}^{*}L_i$,
$u=-\sum_{i=1}^{m}\omega_iL_{i}^*r_i$, and 
$\gamma=\sum_{i=1}^{m}\omega_i\|r_i\|^{2}/2$.
Hence, \eqref{e:id3} follows from \eqref{e:id2}.
\end{proof}

\subsection{Demiregularity}

\begin{definition}{\rm\cite[Definition~2.3]{Sico10}}
\label{d:demir}
An operator $A\colon\HH\to 2^{\HH}$ is \emph{demiregular} at
$x\in\dom A$ if, for every sequence $((x_n,u_n))_{n\in\NN}$ in 
$\gra A$ and every $u\in Ax$ such that $x_n\weakly x$ and 
$u_n\to u$, we have $x_n\to x$.
\end{definition}

\begin{lemma}{\rm\cite[Proposition~2.4]{Sico10}}
\label{l:2009-09-20}
Let $A\colon\HH\to 2^{\HH}$ be monotone and suppose that 
$x\in\dom A$. Then $A$ is demiregular at $x$ in each of the 
following cases.
\begin{enumerate}
\item
\label{l:2009-09-20i}
$A$ is uniformly monotone at $x$, i.e., there exists 
an increasing function $\phi\colon\RP\to\RPX$ that vanishes only 
at $0$ such that
$(\forall u\in Ax)(\forall (y,v)\in\gra A)$
$\scal{x-y}{u-v}\geq\phi(\|x-y\|)$.
\item
\label{l:2009-09-20ii}
$A$ is strongly monotone, i.e., there exists $\alpha\in\RPP$ such
that $A-\alpha\Id$ is monotone.
\item
\label{l:2009-09-20iv-}
$J_A$ is compact, i.e., for every bounded set $C\subset\HH$,
the closure of $J_A(C)$ is compact. In particular, 
$\dom A$ is boundedly relatively compact, i.e., the intersection of 
its closure with every closed ball is compact.
\item
\label{l:2009-09-20vi}
$A\colon\HH\to\HH$ is single-valued with a single-valued continuous
inverse.
\item
\label{l:2009-09-20vii}
$A$ is single-valued on $\dom A$ and $\Id-A$ is demicompact, i.e., 
for every bounded sequence $(x_n)_{n\in\NN}$ in $\dom A$ such 
that $(Ax_n)_{n\in\NN}$ converges strongly, $(x_n)_{n\in\NN}$ 
admits a strong cluster point.
\item
\label{p:2009-09-20ii+}
$A=\partial f$, where $f\in\Gamma_0(\HH)$ is uniformly convex
at $x$, i.e., there exists an increasing function 
$\phi\colon\RP\to\RPX$ that vanishes only at $0$ such that
$(\forall\alpha\in\zeroun)(\forall y\in\dom f)$
$f\big(\alpha x+(1-\alpha) y\big)+\alpha(1-\alpha)\phi(\|x-y\|)
\leq \alpha f(x)+(1-\alpha)f(y)$.
\item
\label{p:2009-09-20ii++++}
$A=\partial f$, where $f\in\Gamma_0(\HH)$ and, for every
$\xi\in\RR$, $\menge{x\in\HH}{f(x)\leq\xi}$ is boundedly compact.
\end{enumerate}
\end{lemma}

\section{Algorithm and convergence}
\label{sec:4}

Our main result is stated in the following theorem.

\begin{theorem}
\label{t:1}
Let $A\colon\HH\to 2^{\HH}$ be maximally monotone, let
$\alpha\in\RPP$, let $\beta\in\RPP$, let $B\colon\HH\to\HH$ be 
$\beta$-cocoercive, let $(\eta_n)_{n\in\NN}\in\ell_+^1(\NN)$,
and let $(U_n)_{n\in\NN}$ be a sequence in 
$\BP_{\alpha}(\HH)$ such that
\begin{equation}
\label{e:palawan2012-03-09}
\mu=\sup_{n\in\NN}\|U_n\|<\pinf\quad\text{and}\quad
(\forall n\in\NN)\quad (1+\eta_n)U_{n+1} \succcurlyeq U_n.
\end{equation}
Let $\varepsilon\in\left]0,\min\{1,2\beta/(\mu+1)\}\right]$,
let $(\lambda_n)_{n\in\NN}$ be a sequence in 
$\left[\varepsilon,1\right]$, let $(\gamma_n)_{n\in\NN}$ be a 
sequence in $[\varepsilon,(2\beta-\varepsilon)/\mu]$,  let
$x_0\in\HH$, and let $(a_n)_{n\in\NN}$ and $(b_n)_{n\in\NN}$ be 
absolutely summable sequences in $\HH$. Suppose that
\begin{equation}
\label{e:palawan2012-03-10}
Z=\zer(A+B)\neq\emp,
\end{equation}
and set 
\begin{equation}
\label{e:forward}
(\forall n\in\NN)\quad
\begin{array}{l}
\left\lfloor
\begin{array}{l}
y_n= x_n-\gamma_n U_n (Bx_n+b_n)\\[1mm]
x_{n+1}=x_n+\lambda_n\big(J_{\gamma_nU_nA}\,
(y_n)+a_n-x_n\big).
\end{array}
\right.\\[2mm]
\end{array}
\end{equation}
Then the following hold for some $\overline{x}\in Z$.
\begin{enumerate}
\item 
\label{t:1i}
$x_n\weakly\overline{x}$.
\item 
\label{t:1ii}
$\sum_{n\in\NN}\|Bx_n-B\overline{x}\|^2<\pinf$.
\item
\label{t:1iii}
Suppose that one of the following holds.
\begin{enumerate}
\item
\label{t:1iiia}
$\varliminf d_Z(x_n)=0$.
\item
\label{t:1iiib}
At every point in $Z$, $A$ or $B$ is demiregular 
(see Lemma~\ref{l:2009-09-20} for special cases).
\item
\label{t:1iiic}
$\inte Z\neq\emp$ and there exists 
$(\nu_n)_{n\in\NN}\in\ell_+^1(\NN)$ such that 
$(\forall n\in\NN)$ $(1+\nu_n)U_n\succcurlyeq U_{n+1}$.
\end{enumerate}
Then $x_n\to\overline{x}$. 
\end{enumerate}
\end{theorem}
\begin{proof}
Set 
\begin{equation}
\label{e:2}
(\forall n\in\NN)\quad
\begin{cases}
A_n=\gamma_nU_nA\\
B_n=\gamma_nU_nB
\end{cases}
\quad\text{and}\quad
\begin{cases}
p_n=J_{A_n}y_n\\
q_n=J_{A_n}(x_n-B_nx_n)\\
s_n=x_n+\lambda_n(q_n-x_n).
\end{cases}
\end{equation}
Then \eqref{e:forward} can be written as
\begin{equation}
\label{e:forward1}
(\forall n\in\NN)\quad
x_{n+1}=x_n+\lambda_n(p_n+a_n-x_n).
\end{equation}
On the other hand, \eqref{e:palawan2012-03-09} 
and Lemma~\ref{l:kjMMXII}\ref{l:kjMMXII-i}\&\ref{l:kjMMXII-iii} 
yield 
\begin{equation}
\label{e:sm-north2012-03-13a}
(\forall n\in\NN)\quad\|U_n^{-1}\|\leq\frac{1}{\alpha},\quad
U_n^{-1}\in{\EuScript P}_{1/\mu}(\HH),\quad\text{and}\quad
(1+\eta_n)U_n^{-1}\succcurlyeq U_{n+1}^{-1}
\end{equation}
and, therefore,
\begin{equation}
\label{e:balay2012-03-13a}
(\forall n\in\NN)(\forall x\in\HH)\quad
(1+\eta_n)\|x\|^2_{U^{-1}_n}\geq\|x\|^2_{U_{n+1}^{-1}}.
\end{equation}
Hence, we derive from \eqref{e:forward1}, \eqref{e:2}, 
Lemma~\ref{l:maxmon45}\ref{l:maxmon45ii}, 
\eqref{e:sm-north2012-03-13a} and 
\eqref{e:palawan2012-03-09} that
\begin{align}
\label{e:3}
(\forall n\in\NN)\quad\|x_{n+1}-s_n\|_{U^{-1}_n}
&\leq\lambda_n\Big(\|a_n\|_{U^{-1}_n}+\|p_n-q_n\|_{U_n^{-1}}\Big) 
\nonumber\\
&\leq\|a_n\|_{U^{-1}_n}+
\|y_n-x_n+B_nx_n\|_{U_n^{-1}}\nonumber\\
&\leq\|a_n\|_{U^{-1}_n}+\gamma_n\|U_nb_n\|_{U_n^{-1}}\nonumber\\
&\leq\sqrt{\|U_n^{-1}\|}\,\|a_n\|+
\gamma_n\sqrt{\|U_n\|}\,\|b_n\|\nonumber\\
&\leq\frac{1}{\sqrt{\alpha}}\|a_n\|+
\frac{2\beta-\varepsilon}{\sqrt{\mu}}\|b_n\|.
\end{align}
Now let $z\in Z$. Since $B$ is $\beta$-cocoercive, 
\begin{equation}
\label{e:4} 
(\forall n\in\NN)\quad\scal{x_n-z}{Bx_n-Bz}\geq\beta\|Bx_n-Bz\|^2.
\end{equation}
On the other hand, it follows from \eqref{e:palawan2012-03-09} that  
\begin{equation}
\label{e:5}
\big(\forall n\in\NN\big)\quad\|B_nx_n-B_n z\|_{U^{-1}_n}^2\leq
\gamma_n^2\|U_n\|\,\|Bx_n-Bz\|^2\leq\gamma_n^2\mu\|Bx_n-Bz\|^2.
\end{equation}
We also note that, since $-Bz\in Az$, \eqref{e:2} yields
\begin{equation}
\label{e:forward2}
\big(\forall n\in\NN\big)\quad z= J_{A_n}(z-B_nz).
\end{equation}
Altogether, it follows from \eqref{e:2}, \eqref{e:forward2}, 
Lemma~\ref{l:maxmon45}\ref{l:maxmon45ii}, \eqref{e:4}, and 
\eqref{e:5} that
\begin{align}
\label{e:7}
(\forall n\in\NN)\quad\|q_n-z\|_{U^{-1}_n}^2
&\leq\|(x_n-z)-(B_nx_n- B_n z)\|_{U^{-1}_n}^2\nonumber\\
&\quad\;-\|(x_n-q_n)-(B_nx_n- B_n z)\|_{U^{-1}_n}^2\nonumber\\
&=\|x_n-z\|_{U^{-1}_n}^2-2\scal{x_n-z}{B_nx_n-B_nz}_{U^{-1}_n}
+\|B_nx_n-B_nz\|_{U^{-1}_n}^2\nonumber\\
&\quad\;-\|(x_n-q_n)-(B_nx_n- B_n z)\|_{U^{-1}_n}^2\nonumber\\
&=\|x_n-z\|_{U^{-1}_n}^2-2\gamma_n\scal{x_n-z}{Bx_n- Bz}+ 
\|B_nx_n-B_n z\|_{U^{-1}_n}^2\nonumber\\
&\quad\;-\|(x_n-q_n)-(B_nx_n- B_n z)\|_{U^{-1}_n}^2\nonumber\\
&\leq\|x_n-z\|_{U^{-1}_n}^2-\gamma_n(2\beta-\mu\gamma_n)
\|Bx_n-Bz\|^2\nonumber\\
&\quad\;-\|(x_n-q_n)-(B_nx_n- B_n z)\|_{U^{-1}_n}^2\nonumber\\
&\leq\|x_n-z\|_{U^{-1}_n}^2-\varepsilon^2\|Bx_n-Bz\|^2\nonumber\\
&\quad\;-\|(x_n-q_n)-(B_nx_n- B_n z)\|_{U^{-1}_n}^2.
\end{align}
In turn, we derive from \eqref{e:balay2012-03-13a} and \eqref{e:2} 
that
\begin{align}
\label{e:ab2}
(\forall n\in\NN)\quad 
(1+\eta_n)^{-1}\|s_n-z\|_{U^{-1}_{n+1}}^2 
&\leq\|s_n-z\|_{U^{-1}_n}^2 \nonumber\\
&\leq(1-\lambda_n)\|x_n-z\|_{U^{-1}_n}^2 
+\lambda_n\|q_n-z\|_{U^{-1}_n}^2\nonumber\\
&\leq\|x_n-z\|_{U^{-1}_n}^2 
-\varepsilon^3\|Bx_n-Bz\|^2\nonumber\\
&\quad\;-\varepsilon\|(x_n-q_n)-(B_nx_n- B_n z)\|_{U^{-1}_n}^2,
\end{align}
which implies that 
\begin{align}
(\forall n\in\NN)\quad\|s_n-z\|_{U^{-1}_{n+1}}^2
&\leq(1+\eta_n)\|x_n-z\|_{U^{-1}_n}^2
-\varepsilon^3\|Bx_n-Bz\|^2\nonumber\\
&\quad\;-\varepsilon\|(x_n-q_n)-(B_nx_n- B_n z)\|_{U^{-1}_n}^2
\label{e:up-diliman2012-03-13a}\\
&\leq\delta^2\|x_n-z\|_{U^{-1}_n}^2,
\label{e:up-diliman2012-03-13b}
\end{align}
where
\begin{equation}
\label{e:balay2012-03-13k}
\delta=\sup_{n\in\NN}\sqrt{1+\eta_n}.
\end{equation}
Next, we set
\begin{equation}
\label{e:balay2012-03-13b}
(\forall n\in\NN)\quad\varepsilon_n=
\delta\bigg(\frac{1}{\sqrt{\alpha}}\|a_n\|+
\frac{2\beta-\varepsilon}{\sqrt{\mu}}\|b_n\|\bigg).
\end{equation}
Then our assumptions yield
\begin{equation}
\label{e:balay2012-03-13c}
\sum_{n\in\NN}\varepsilon_n<\pinf.
\end{equation}
Moreover, using \eqref{e:balay2012-03-13a},
\eqref{e:up-diliman2012-03-13a}, and 
\eqref{e:3}, we obtain
\begin{align}
\label{e:ab3}
(\forall n\in\NN)\quad 
\|x_{n+1}-z\|_{U_{n+1}^{-1}}
&\leq\|x_{n+1}-s_n\|_{U_{n+1}^{-1}}
+\|s_n-z\|_{U_{n+1}^{-1}}\nonumber\\
&\leq\sqrt{1+\eta_n}\|x_{n+1}-s_n\|_{U_n^{-1}}
+\sqrt{1+\eta_n}\|x_n-z\|_{U_n^{-1}}\nonumber\\
&\leq\delta\|x_{n+1}-s_n\|_{U_n^{-1}}+
\sqrt{1+\eta_n}\|x_n-z\|_{U_n^{-1}}\nonumber\\
&\leq\sqrt{1+\eta_n}\|x_n-z\|_{U_n^{-1}}+\varepsilon_n\nonumber\\
&\leq(1+\eta_n)\|x_n-z\|_{U_n^{-1}}+\varepsilon_n.
\end{align}
In view of \eqref{e:sm-north2012-03-13a}, \eqref{e:balay2012-03-13c},
and \eqref{e:ab3}, we can apply Proposition~\ref{p:1} to assert 
that $(\|x_n-z\|_{U_n^{-1}})_{n\in\NN}$ converges and, 
therefore, that
\begin{equation}
\label{e:balay2012-03-13i}
\zeta=\sup_{n\in\NN}\|x_n-z\|_{U_n^{-1}}<\pinf.
\end{equation}
On the other hand, \eqref{e:balay2012-03-13a}, \eqref{e:3}, and 
\eqref{e:balay2012-03-13b} yield
\begin{equation}
\label{e:balay2012-03-13g}
(\forall n\in\NN)\quad
\|x_{n+1}-s_n\|^2_{U_{n+1}^{-1}}\leq(1+\eta_n)
\|x_{n+1}-s_n\|^2_{U_n^{-1}}\leq\varepsilon^2_n.
\end{equation}
Hence, using \eqref{e:up-diliman2012-03-13a},
\eqref{e:up-diliman2012-03-13b}, \eqref{e:balay2012-03-13k}, and 
\eqref{e:balay2012-03-13i}, we get
\begin{align}
\label{e:balay2012-03-13f}
(\forall n\in\NN)\quad\|x_{n+1}-z\|^2_{U_{n+1}^{-1}}
&\leq\|s_n-z\|^2_{U_{n+1}^{-1}}+2\|s_n-z\|_{U_{n+1}^{-1}}\,
\|x_{n+1}-s_n\|_{U_{n+1}^{-1}}+\|x_{n+1}-s_n\|^2_{U_{n+1}^{-1}}
\nonumber\\
&\leq(1+\eta_n)\|x_n-z\|_{U^{-1}_n}^2
-\varepsilon^3\|Bx_n-Bz\|^2\nonumber\\
&\quad\;-\varepsilon\|x_n-q_n-B_nx_n+B_nz\|_{U^{-1}_n}^2
+2\delta\zeta\varepsilon_n+\varepsilon_n^2\nonumber\\
&\leq\|x_n-z\|_{U^{-1}_n}^2-\varepsilon^3\|Bx_n-Bz\|^2
-\varepsilon\|x_n-q_n-B_nx_n+B_nz\|_{U^{-1}_n}^2
\nonumber\\
&\quad\;+\zeta^2\eta_n+2\delta\zeta\varepsilon_n+\varepsilon_n^2.
\end{align}
Consequently, for every $N\in\NN$,
\begin{align}
\label{e:balay2012-03-13m}
\varepsilon^3\sum_{n=0}^N\|Bx_n-Bz\|^2
&\leq\|x_0-z\|_{U^{-1}_0}^2-
\|x_{N+1}-z\|_{U^{-1}_{N+1}}^2
+\sum_{n=0}^N\big(\zeta^2\eta_n+2\delta\zeta\varepsilon_n+
\varepsilon_n^2\big)\nonumber\\
&\leq\zeta^2+\sum_{n=0}^N\big(\zeta^2\eta_n+2\delta\zeta\varepsilon_n
+\varepsilon_n^2\big).
\end{align}
Appealing to \eqref{e:balay2012-03-13c} and the summability of
$(\eta_n)_{n\in\NN}$, taking the limit as $N\to\pinf$, yields 
\begin{equation}
\label{e:balay2012-03-13r}
\sum_{n\in\NN}\|Bx_n-Bz\|^2\leq\frac{1}{\varepsilon^3}\bigg(
\zeta^2+\sum_{n\in\NN}\big(\zeta^2\eta_n+
2\delta\zeta\varepsilon_n+\varepsilon_n^2\big)\bigg)<\pinf.
\end{equation}
We likewise derive from \eqref{e:balay2012-03-13f} that
\begin{equation}
\label{e:klm2012-03-14a}
\sum_{n\in\NN}\big\|x_n-q_n-B_nx_n+B_nz\big\|_{U^{-1}_n}^2
<\pinf.
\end{equation}

\ref{t:1i}: Let $x$ be a weak sequential cluster point of 
$(x_n)_{n\in\NN}$, say $x_{k_n}\weakly x$. In view of 
\eqref{e:ab3}, \eqref{e:sm-north2012-03-13a}, and 
Proposition~\ref{p:guad1}, it is enough to show that $x\in Z$. 
On the one hand, \eqref{e:balay2012-03-13r} yields 
$Bx_{k_n}\to Bz$. On the other hand, since 
$B$ is cocoercive, it is maximally monotone 
\cite[Example~20.28]{Livre1} and its graph is therefore 
sequentially closed in $\HH^{\text{weak}}\times\HH^{\text{strong}}$
\cite[Proposition~20.33(ii)]{Livre1}. This implies that
$Bx=Bz$ and hence that $Bx_{k_n}\to Bx$. Thus, in view of
\eqref{e:balay2012-03-13r},
\begin{equation}
\label{e:klm2012-03-14z}
\sum_{n\in\NN}\|Bx_n-Bx\|^2<\pinf.
\end{equation}
Now set
\begin{equation}
\label{e:klm2012-03-14c}
(\forall n\in\NN)\quad u_n=
\frac{1}{\gamma_n}U_n^{-1}(x_n-q_n)-Bx_n.
\end{equation}
Then it follows from \eqref{e:2} that
\begin{equation}
\label{e:klm2012-03-14d}
(\forall n\in\NN)\quad u_n\in Aq_n.
\end{equation}
In addition, 
\eqref{e:2}, \eqref{e:sm-north2012-03-13a}, and
\eqref{e:klm2012-03-14a} yield
\begin{align}
\label{e:klm2012-03-14E}
\|u_n+Bx\|
&=\frac{1}{\gamma_n}\|U_n^{-1}(x_n-q_n-B_nx_n+B_nx)\|
\nonumber\\
&\leq\frac{1}{\varepsilon\alpha}\|x_n-q_n-B_nx_n+B_nx\|
\nonumber\\
&\leq\frac{\sqrt{\mu}}{\varepsilon\alpha}
\|x_n-q_n-B_nx_n+B_nx\|_{U_n^{-1}}
\nonumber\\
&\to 0.
\end{align}
Moreover, it follows from \eqref{e:2}, \eqref{e:palawan2012-03-09},
and \eqref{e:klm2012-03-14z} that
\begin{align}
\label{e:klm2012-03-14e}
\|x_n-q_n\|
&\leq\|x_n-q_n-B_nx_n+B_nx\|+\|B_nx_n-B_nx\|\nonumber\\
&\leq\|x_n-q_n-B_nx_n+B_nx\|+\gamma_n\|U_n\|\,\|Bx_n-Bx\|
\nonumber\\
&\leq\|x_n-q_n-B_nx_n+B_nx\|+(2\beta-\varepsilon)
\|Bx_n-Bx\|\nonumber\\
&\to 0
\end{align}
and, therefore, since $x_{k_n}\weakly x$, that $q_{k_n}\weakly x$. 
To sum up,
\begin{equation}
\label{e:klm2012-03-14f}
q_{k_n}\weakly x,\quad u_{k_n}\to -Bx,\quad\text{and}\quad
(\forall n\in\NN)\quad (q_{k_n},u_{k_n})\in\gra A.
\end{equation}
Hence, using the sequential closedness of $\gra A$ in
$\HH^{\text{weak}}\times\HH^{\text{strong}}$
\cite[Proposition~20.33(ii)]{Livre1}, we conclude that
$-Bx\in Ax$, i.e., that $x\in Z$.

\ref{t:1ii}: Since $\overline{x}\in Z$, the claim follows from 
\eqref{e:balay2012-03-13r}.

\ref{t:1iii}: We now prove strong convergence.

\ref{t:1iiia}: Since $A$ and $B$ are maximally monotone and 
$\dom B=\HH$, $A+B$ is maximally monotone 
\cite[Corollary~24.4(i)]{Livre1} and $Z$ is therefore 
closed \cite[Proposition~23.39]{Livre1}. Hence, the claim follows
from \ref{t:1i}, \eqref{e:ab3}, and Proposition~\ref{p:monodc}.

\ref{t:1iiib}: It follows from \ref{t:1i} and 
\eqref{e:klm2012-03-14e} that $q_n\weakly\overline{x}\in Z$
and from \eqref{e:klm2012-03-14E} that 
$u_n\to -B\overline{x}\in A\overline{x}$. Hence, if $A$ is 
demiregular at $\overline{x}$, \eqref{e:klm2012-03-14d} yields
$q_n\to\overline{x}$. In view of \eqref{e:klm2012-03-14e}, we
conclude that $x_n\to\overline{x}$. Now suppose that $B$ is 
demiregular at $\overline{x}$. Then since 
$x_n\weakly\overline{x}\in Z$ by \ref{t:1i} and 
$Bx_n\to B\overline{x}$ by \ref{t:1ii}, we conclude that
$x_n\to\overline{x}$. 

\ref{t:1iiic}: Suppose that $z\in\inte C$ and fix $\rho\in\RPP$ 
such that $B(z;\rho)\subset C$. It follows from 
\eqref{e:balay2012-03-13i} that 
$\theta=\sup_{x\in B(z;\rho)}\sup_{n\in\NN}
\|x_n-x\|_{U^{-1}_n}\leq(1/\sqrt{\alpha})(\sup_{n\in\NN}
\|x_n-z\|+\sup_{x\in B(z;\rho)}\|x-z\|)<\pinf$ and 
from \eqref{e:balay2012-03-13f} that 
\begin{align}
\label{e:balay2012-03-13ff}
(\forall n\in\NN)(\forall x\in B(z;\rho))\quad
\|x_{n+1}-x\|^2_{U_{n+1}^{-1}}
&\leq\|x_n-x\|_{U^{-1}_n}^2
+\theta^2\eta_n+2\delta\theta\varepsilon_n+\varepsilon_n^2.
\end{align}
Hence, the claim follows from \ref{t:1i}, Lemma~\ref{l:kjMMXII},
and Proposition~\ref{p:2012-03-26}.
\end{proof}

\begin{remark} 
Here are some observations on Theorem~\ref{t:1}.
\begin{enumerate}
\item
Suppose that $(\forall n\in\NN)$ $U_n=\Id$. Then \eqref{e:forward} 
relapses to the forward-backward algorithm studied in 
\cite{Sico10,Opti04}, which itself captures those of 
\cite{Lema97,Merc79,Tsen91}. Theorem~\ref{t:1} extends the
convergence results of these papers. 
\item
As shown in \cite[Remark~5.12]{Smms05}, the convergence of the 
forward-backward iterates to a solution may be only weak and 
not strong, hence the necessity of the additional conditions in 
Theorem~\ref{t:1}\ref{t:1iii}.
\item
In Euclidean spaces, condition \eqref{e:palawan2012-03-09} 
was used in \cite{Pare08} in a variable metric proximal point
algorithm and then in \cite{Loli09} in a more general splitting
algorithm.
\end{enumerate}
\end{remark}

Next, we describe direct applications of Theorem~\ref{t:1},
which yield new variable metric splitting schemes. We start with 
minimization problems, an area in which the forward-backward 
algorithm has found numerous applications, e.g., 
\cite{Banf11,Smms05,Duch09,Tsen90,Tsen91}.

\begin{example}
\label{ex:2012-04-19}
Let $f\in\Gamma_0(\HH)$, let $\alpha\in\RPP$, let $\beta\in\RPP$, 
let $g\colon\HH\to\RR$ be convex and differentiable with a 
$1/\beta$-Lipschitzian gradient, let 
$(\eta_n)_{n\in\NN}\in\ell_+^1(\NN)$, and let $(U_n)_{n\in\NN}$ be 
a sequence in $\BP_{\alpha}(\HH)$ such that
\eqref{e:palawan2012-03-09} holds. Furthermore, let 
$\varepsilon\in\left]0,\min\{1,2\beta/(\mu+1)\}\right]$ where 
$\mu$ is given by \eqref{e:palawan2012-03-09}, let 
$(\lambda_n)_{n\in\NN}$ be a sequence in 
$\left[\varepsilon,1\right]$, let $(\gamma_n)_{n\in\NN}$ be a 
sequence in $[\varepsilon,(2\beta-\varepsilon)/\mu]$,  let
$x_0\in\HH$, and let $(a_n)_{n\in\NN}$ and $(b_n)_{n\in\NN}$ be 
absolutely summable sequences in $\HH$. Suppose that
$\Argmin\:(f+g)\neq\emp$ and set 
\begin{equation}
\label{e:forward6}
(\forall n\in\NN)\quad
\begin{array}{l}
\left\lfloor
\begin{array}{l}
y_n= x_n-\gamma_n U_n (\nabla g(x_n)+b_n)\\[1mm]
x_{n+1}=x_n+\lambda_n\big(\prox_{\gamma_n f}^{U^{-1}_n}\,
y_n+a_n-x_n\big).
\end{array}
\right.\\[2mm]
\end{array}
\end{equation}
Then the following hold for some 
$\overline{x}\in\Argmin\:(f+g)$.
\begin{enumerate}
\item 
\label{ex:2012-04-19i}
$x_n\weakly\overline{x}$.
\item 
\label{ex:2012-04-19ii}
$\sum_{n\in\NN}\|\nabla g(x_n)-\nabla g(\overline{x})\|^2<\pinf$.
\item
\label{ex:2012-04-19iii}
Suppose that one of the following holds.
\begin{enumerate}
\item
$\varliminf d_{\Argmin\:(f+g)}(x_n)=0$.
\item
At every point in $\Argmin\:(f+g)$, $f$ or $g$ is uniformly
convex (see Lemma~\ref{l:2009-09-20}\ref{p:2009-09-20ii+}).
\item
$\inte\Argmin\:(f+g)\neq\emp$ and there exists 
$(\nu_n)_{n\in\NN}\in\ell_+^1(\NN)$ such that 
$(\forall n\in\NN)$ $(1+\nu_n)U_n\succcurlyeq U_{n+1}$.
\end{enumerate}
Then $x_n\to\overline{x}$. 
\end{enumerate}
\end{example}
\begin{proof}
An application of Theorem~\ref{t:1} with $A=\partial f$ and 
$B=\nabla g$, since the Baillon-Haddad theorem
\cite[Corollary~18.16]{Livre1} ensures that $\nabla g$ is 
$\beta$-cocoercive and since, by \cite[Corollary~26.3]{Livre1},
$\Argmin\:(f+g)=\zer(A+B)$.
\end{proof}

The next example addresses variational inequalities, another area 
of application of forward-backward splitting 
\cite{Livre1,Facc03,Tsen90,Tsen91}.

\begin{example}
\label{ex:2012-04-17}
Let $f\in\Gamma_0(\HH)$, let $\alpha\in\RPP$, let $\beta\in\RPP$, 
let $B\colon\HH\to\HH$ be $\beta$-cocoercive, let 
$(\eta_n)_{n\in\NN}\in\ell_+^1(\NN)$, and let $(U_n)_{n\in\NN}$ be
a sequence in $\BP_{\alpha}(\HH)$ that satisfies 
\eqref{e:palawan2012-03-09}. Furthermore,
let $\varepsilon\in\left]0,\min\{1,2\beta/(\mu+1)\}\right]$ where 
$\mu$ is given by \eqref{e:palawan2012-03-09}, let 
$(\lambda_n)_{n\in\NN}$ be a sequence in 
$\left[\varepsilon,1\right]$, let $(\gamma_n)_{n\in\NN}$ be a 
sequence in $[\varepsilon,(2\beta-\varepsilon)/\mu]$,  let
$x_0\in\HH$, and let $(a_n)_{n\in\NN}$ and $(b_n)_{n\in\NN}$ be 
absolutely summable sequences in $\HH$. Suppose that the 
variational inequality
\begin{equation}
\label{e:2012-04-18e}
\text{find}\quad x\in\HH\quad\text{such that}\quad
(\forall y\in\HH)\quad\scal{x-y}{Bx}+f(x)\leq f(y)
\end{equation}
admits at least one solution and set
\begin{equation}
\label{e:vforward}
(\forall n\in\NN)\quad
\begin{array}{l}
\left\lfloor
\begin{array}{l}
y_n= x_n-\gamma_n U_n (Bx_n+b_n)\\[1mm]
x_{n+1}=x_n+\lambda_n\big(\prox_{\gamma_n f}^{U^{-1}_n}\,
y_n+a_n-x_n\big).
\end{array}
\right.\\[2mm]
\end{array}
\end{equation}
Then $(x_n)_{n\in\NN}$ converges weakly to a solution 
$\overline{x}$ to \eqref{e:2012-04-18e}.
\end{example}
\begin{proof}
Set $A=\partial f$ in Theorem~\ref{t:1}\ref{t:1i}.
\end{proof}

\section{Strongly monotone inclusions in duality}
\label{sec:5}
In \cite{Svva10}, strongly convex composite minimization 
problems of the form
\begin{equation}
\label{e:prob1}
\minimize{x\in\HH}f(x)+g(Lx-r)+\frac12\|x-z\|^2,
\end{equation}
where $z\in\HH$, $r\in\GG$, $f\in\Gamma_0(\HH)$, 
$g\in\Gamma_0(\GG)$, and $L\in\BL(\HH,\GG)$, were solved by 
applying the forward-backward algorithm to the Fenchel-Rockafellar 
dual problem
\begin{equation}
\label{e:prob2}
\minimize{v\in\GG}\widetilde{f^*}(z-L^*v)+g^*(v)+\scal{v}{r},
\end{equation}
where $\widetilde{f^*}=f^*\infconv(\|\cdot\|^2/2)$ denotes the 
Moreau envelope of $f^*$. This framework was shown to capture and
extend various formulations in areas such as sparse signal recovery, 
best approximation theory, and inverse problems. In this section, 
we use the results of Section~\ref{sec:4} to generalize this 
framework in several directions simultaneously. 
First, we consider general monotone inclusions, not just 
minimization problems. Second, we incorporate parallel sum 
components (see \eqref{e:parasum}) in the model. Third, our 
algorithm allows for a variable metric. The following problem is 
formulated using the duality framework of \cite{Svva12}, which 
itself extends those of 
\cite{Atto96,Ecks99,Mosc72,Penn00,Robi99,Rock67}.

\begin{problem}
\label{prob:5.1}
Let $z\in\HH$, let $\rho\in\RPP$, let $A\colon\HH\to 2^{\HH}$ be 
maximally monotone, and let $m$ be a strictly positive integer. 
For every $i\in\{1,\ldots,m\}$, let $r_i\in\GG_i$, let 
$B_i\colon\GG_i\to 2^{\GG_i}$ be maximally monotone, let
$\nu_i\in\RPP$, let $D_i\colon\GG_i\to 2^{\GG_i}$ be maximally 
monotone and $\nu_i$-strongly monotone, and suppose that 
$0\neq L_i\in\BL(\HH,\GG_i)$. Furthermore, suppose that
\begin{equation}
\label{e:2012-03-26a}
z\in\ran\bigg(A+\sum_{i=1}^mL_i^*\big((B_i\infconv D_i)
(L_i\cdot-r_i)\big)+\rho\Id\bigg).
\end{equation}
The problem is to solve the primal inclusion
\begin{equation}
\label{e:fprimal}
\text{find}\;\;\overline{x}\in\HH\;\;\text{such that}\;\;
z\in A\overline{x}+\sum_{i=1}^mL_i^*\big((B_i\infconv D_i)
(L_i\overline{x}-r_i)\big)+\rho\overline{x},
\end{equation}
together with the dual inclusion
\begin{multline}
\label{e:fdual}
\text{find}\;\;\overline{v_1}\in\GG_1,\:\ldots,\:
\overline{v_m}\in\GG_m\;\:\text{such that}\\
(\forall i\in\{1,\ldots,m\})\quad r_i\in
L_i\bigg(J_{\rho^{-1}A}\bigg(\rho^{-1}\bigg(
z-\sum_{j=1}^mL_j^*\overline{v_j}\bigg)\bigg)\bigg)-
B_i^{-1}\overline{v_i}-D_i^{-1}\overline{v_i}.
\end{multline}
\end{problem}

Let us start with some properties of Problem~\ref{prob:5.1}.

\begin{proposition}
\label{p:2012-03-20}
In Problem~\ref{prob:5.1}, set
\begin{equation}
\label{e:2012-03-21a}
\overline{x}=J_{\rho^{-1}M}\big(\rho^{-1}z\big),
\quad\text{where}\quad
M=A+\sum_{i=1}^mL_i^*\circ(B_i\infconv D_i)\circ(L_i\cdot-r_i).
\end{equation}
Then the following hold.
\begin{enumerate}
\item
\label{p:2012-03-20i}
$\overline{x}$ is the unique solution to the primal problem 
\eqref{e:fprimal}.
\item
\label{p:2012-03-20ii}
The dual problem \eqref{e:fdual} admits at least one solution.
\item
\label{p:2012-03-20iii}
Let $(\overline{v_1},\ldots,\overline{v_m})$ be a solution to 
\eqref{e:fdual}. Then 
$\overline{x}=J_{\rho^{-1}A}\big(\rho^{-1}\big(z-
\sum_{i=1}^mL_i^*\overline{v_i}\big)\big)$.
\item
\label{p:2012-03-20iv}
Condition \eqref{e:2012-03-26a} is satisfied for every $z$ 
in $\HH$ if and only if $M$ is maximally monotone. 
This is true when one of the following holds.
\begin{enumerate}
\item
\label{p:2012-03-20iva}
The conical hull of 
\begin{equation}
\label{e:2012-03-28b}
E=\Menge{\big(L_ix-r_i-v_i\big)_{1\leq i\leq m}}
{x\in\dom A\;\text{and}\;(v_i)_{1\leq i\leq m}
\in\underset{i=1}{\overset{m}{\cart}}
\ran\big(B_i^{-1}+D_i^{-1}\big)}
\end{equation}
is a closed vector subspace.
\item
\label{p:2012-03-20ivb}
$A=\partial f$ for some $f\in\Gamma_0(\HH)$, for every 
$i\in\{1,\ldots,m\}$, $B_i=\partial g_i$ for some 
$g_i\in\Gamma_0(\GG_i)$ and $D_i=\partial\ell_i$ for some
strongly convex function $\ell_i\in\Gamma_0(\GG_i)$, and
one of the following holds.
\begin{enumerate}
\item
\label{p:heidelberg2011-07-06i}
$(r_1,\ldots,r_m)\in$ $\sri\big\{(L_ix-y_i)_{1\leq i\leq m}\:\mid\:
x\in\dom f\;\text{and}\;\\
~\hfill (\forall i\in\{1,\ldots,m\})\;\:y_i\in
\dom g_i+\dom\ell_i\big\}$.
\item
\label{p:heidelberg2011-07-06ii}
For every $i\in\{1,\ldots,m\}$, $g_i$ or $\ell_i$ is real-valued.
\item
\label{p:heidelberg2011-07-06iii}
$\HH$ and $(\GG_i)_{1\leq i\leq m}$ are finite-dimensional, 
and there exists $x\in\reli\dom f$ such that 
\begin{equation}
\label{e:gennad2011-08-06b}
(\forall i\in\{1,\ldots,m\})\quad
L_ix-r_i\in\reli\dom g_i+\reli\dom\ell_i.
\end{equation}
\end{enumerate}
\end{enumerate}
\end{enumerate}
\end{proposition}
\begin{proof}
\ref{p:2012-03-20i}: It follows from our assumptions and
\cite[Proposition~20.10]{Livre1} that $\rho^{-1}M$ is a monotone 
operator. Hence, $J_{\rho^{-1}M}$ is a single-valued operator with
domain $\ran(\Id+\rho^{-1}M)$ \cite[Proposition~23.9(ii)]{Livre1}.
Moreover, \eqref{e:2012-03-26a} $\Leftrightarrow$ 
$\rho^{-1}z\in\ran(\Id+\rho^{-1}M)=\dom J_{\rho^{-1}M}$, and,
in view of \eqref{e:resolvent}, the inclusion in \eqref{e:fprimal} 
is equivalent to $\overline{x}=J_{\rho^{-1}M}(\rho^{-1}z)$.

\ref{p:2012-03-20ii}\&\ref{p:2012-03-20iii}: It follows from 
\eqref{e:resolvent} and \eqref{e:parasum} that
\begin{eqnarray}
\label{e:2012-03-28a}
\ref{p:2012-03-20i}
&\Leftrightarrow&(\exi\overline{v_1}\in\GG_1)\cdots
(\exi\overline{v_m}\in\GG_m)\quad
\begin{cases}
(\forall i\in\{1,\ldots,m\})\quad
\overline{v_i}\in(B_i\infconv D_i)(L_i\overline{x}-r_i)\\
z-\sum_{i=1}^mL_i^*\overline{v_i}\in A\overline{x}+\rho\overline{x}
\end{cases}\nonumber\\
&\Leftrightarrow&(\exi\overline{v_1}\in\GG_1)\cdots
(\exi\overline{v_m}\in\GG_m)\quad
\begin{cases}
(\forall i\in\{1,\ldots,m\})\quad
r_i\in L_i\overline{x}-B_i^{-1}\overline{v_i}-
D_i^{-1}\overline{v_i}\\
\overline{x}=J_{\rho^{-1}A}\big(\rho^{-1}\big(z-\sum_{j=1}^m
L_j^*\overline{v_j}\big)\big)
\end{cases}\nonumber\\
&\Leftrightarrow&
\begin{cases}
(\overline{v_1},\ldots,\overline{v_m})\;\text{solves}\;
\eqref{e:fdual}\\
\overline{x}=J_{\rho^{-1}A}\big(\rho^{-1}\big(z-\sum_{j=1}^m
L_j^*\overline{v_j}\big)\big).
\end{cases}
\end{eqnarray}

\ref{p:2012-03-20iv}: 
It follows from Minty's theorem \cite[Theorem~21.1]{Livre1}, 
that $M+\rho\Id$ is surjective if and only if $M$ is 
maximally monotone. 

\ref{p:2012-03-20iva}: 
Using Notation~\ref{n:palawan-mai2008-}, let us set
\begin{equation}
\label{e:2012-03-30a}
{\boldsymbol L}\colon\HH\to\GGG\colon{x}\mapsto
\big(L_ix\big)_{1\leq i\leq m}\quad\text{and}\quad
{\boldsymbol B}\colon\GGG\to2^{\GGG}\colon\boldsymbol{y}\mapsto
\big((B_i\infconv D_i)(y_i-r_i)\big)_{1\leq i\leq m}.
\end{equation}
Then it follows from \eqref{e:2012-03-21a} that
$M=A+\boldsymbol{L}^*\circ\boldsymbol{B}\circ\boldsymbol{L}$
and from \eqref{e:2012-03-28b} that 
$E=\boldsymbol{L}(\dom A)-\dom\boldsymbol{B}$. Hence, since 
$\cone(E)=\spc(E)$, in view of \cite[Section~24]{Botr10},
to conclude that $M$ is maximally monotone, it is enough to
show that ${\boldsymbol B}$ is.
For every $i\in\{1,\ldots,m\}$, since $D_i$ is maximally monotone 
and strongly monotone, $\dom D_i^{-1}=\ran D_i=\GG_i$ 
\cite[Proposition~22.8(ii)]{Livre1} and it follows from
\cite[Proposition~20.22~\&~Corollary~24.4(i)]{Livre1} that 
$B_i\infconv D_i$ is maximally monotone. This shows that 
${\boldsymbol B}$ is maximally monotone.

\ref{p:2012-03-20ivb}: This follows from 
\cite[Proposition~4.3]{Svva12}.
\end{proof}

\begin{remark}
\label{r:2012-04-23}
In connection with
Proposition~\ref{p:2012-03-20}\ref{p:2012-03-20iv}, let us note 
that even in the simple setting of normal cone operators in finite 
dimension, some constraint qualification is required to ensure the 
existence of a primal solution for every 
$z\in\HH$. To see this, 
suppose that, in Problem~\ref{prob:5.1}, $\HH$ is the Euclidean 
plane, $m=1$, $\rho=1$, $\GG_1=\HH$, $L_1=\Id$, 
$z=(\zeta_1,\zeta_2)$, $r_1=0$, $D_1=\{0\}^{-1}$, $A=N_C$, and 
$B_1=N_{K}$, where 
$C=\menge{(\xi_1,\xi_2)\in\RR^2}{(\xi_1-1)^2+\xi_2^2\leq 1}$ and
$K=\menge{(\xi_1,\xi_2)\in\RR^2}{\xi_1\leq 0}$. Then 
$\dom(A+B_1+\Id)=\dom A\cap\dom B_1=C\cap K=\{0\}$ and the primal 
inclusion $z\in A\overline{x}+B_1\overline{x}+\overline{x}$
reduces to $(\zeta_1,\zeta_2)\in N_C0+N_{K}0=
\RM\times\{0\}+\RP\times\{0\}=\RR\times\{0\}$, which has no 
solution if $\zeta_2\neq 0$. Here 
$\cone(\dom A-\dom B_1)=\cone(C-K)=-K$ is not a vector subspace.
\end{remark}

In the following result we derive from Theorem~\ref{t:1} a
parallel primal-dual algorithm for solving Problem~\ref{prob:5.1}.

\begin{corollary}
\label{c:2}
In Problem~\ref{prob:5.1}, set
\begin{equation}
\label{e:beta1}
\beta=\frac{1}{\displaystyle{\max_{1\leq i\leq m}\:}
\frac{1}{\nu_i}+\frac{1}{\rho}
\displaystyle{\displaystyle{\sum_{1\leq i\leq m}}\|L_i\|^2}}.
\end{equation}
Let $(a_n)_{n\in\NN}$ be an absolutely summable sequence in $\HH$,
let $\alpha\in\RPP$, and let $(\eta_n)_{n\in\NN}\in\ell_+^1(\NN)$.
For every $i\in\{1,\ldots,m\}$, let $v_{i,0}\in\GG_i$, let 
$(b_{i,n})_{n\in\NN}$ and $(d_{i,n})_{n\in\NN}$ be absolutely 
summable sequences in $\GG_i$, and let $(U_{i,n})_{n\in\NN}$ be 
a sequence in $\BP_{\alpha}(\GG_i)$. Suppose that
\begin{equation}
\label{e:2012-04-08d}
\mu=\max_{1\leq i\leq m}\sup_{n\in\NN}\|U_{i,n}\|<\pinf
\quad\text{and}\quad(\forall i\in\{1,\ldots,m\})(\forall n\in\NN)
\quad(1+\eta_n)U_{i,n+1}\succcurlyeq U_{i,n}.
\end{equation}
Let $\varepsilon\in\left]0,\min\{1,2\beta/(\mu+1)\}\right]$,
let $(\lambda_n)_{n\in\NN}$ be a sequence in 
$\left[\varepsilon,1\right]$, and let $(\gamma_n)_{n\in\NN}$ be a 
sequence in $[\varepsilon,(2\beta-\varepsilon)/\mu]$. Set 
\begin{equation}
\label{e:2012-04-08a}
(\forall n\in\NN)\quad
\left\lfloor
\begin{array}{l}
s_n=z-\sum_{i=1}^mL_i^*v_{i,n}\\
x_n=J_{\rho^{-1}A}(\rho^{-1}s_n)+a_n\\
\operatorname{For}\;i=1,\ldots,m\\
\left\lfloor
\begin{array}{l}
w_{i,n}=v_{i,n}+\gamma_nU_{i,n}\big(L_ix_n-r_i-D_i^{-1}v_{i,n}
-d_{i,n}\big)\\[2mm]
v_{i,n+1}=v_{i,n}+\lambda_n\Big(J_{\gamma_nU_{i,n}B_i^{-1}}
(w_{i,n})+b_{i,n}-v_{i,n}\Big).\\[1mm] 
\end{array}
\right.\\[2mm]
\end{array}
\right.
\end{equation}
Then the following hold for the solution $\overline{x}$ to 
\eqref{e:fprimal} and for some solution 
$(\overline{v_1},\ldots,\overline{v_m})$ to \eqref{e:fdual}.
\begin{enumerate}
\item
\label{c:2i}
$(\forall i\in\{1,\ldots,m\})$ $v_{i,n}\weakly\overline{v_i}$. 
In addition, 
$\overline{x}=J_{\rho^{-1}A}\big(\rho^{-1}\big(z-
\sum_{i=1}^mL_i^*\overline{v_i}\big)\big)$.
\item
\label{c:2ii}
$x_n\to\overline{x}$.
\end{enumerate}
\end{corollary}
\begin{proof}
For every $i\in\{1,\ldots,m\}$, since $D_i$ is maximally monotone 
and $\nu_i$-strongly monotone, $D_i^{-1}$ is $\nu_i$-cocoercive
with $\dom D_i^{-1}=\ran D_i=\GG_i$ 
\cite[Proposition~22.8(ii)]{Livre1}. Let us define 
$\GGG$ as in Notation~\ref{n:palawan-mai2008-}, 
and let us introduce the operators
\begin{equation}
\label{e:2012-04-01a}
\begin{cases}
T\colon\HH\to\HH\colon x\mapsto J_{\rho^{-1}A}
\big(\rho^{-1}(z-x)\big)\\
{\boldsymbol A}\colon\GGG\to2^{\GGG}\colon\boldsymbol{v}\mapsto
\big(B^{-1}_iv_i\big)_{1\leq i\leq m}\\[1mm]
{\boldsymbol D}\colon\GGG\to\GGG\colon\boldsymbol{v}\mapsto
\big(r_i+D^{-1}_iv_i\big)_{1\leq i\leq m}\\[1mm]
{\boldsymbol L}\colon\HH\to\GGG\colon x\mapsto
\big(L_ix\big)_{1\leq i\leq m}
\end{cases}
\end{equation}
and
\begin{equation}
\label{e:2012-04-01b}
(\forall n\in\NN)\quad
{\boldsymbol U}_n\colon\GGG\to\GGG\colon{\boldsymbol v}
\mapsto\big(U_{i,n}v_{i}\big)_{1\leq i\leq m}.
\end{equation}

\ref{c:2i}:
In view of \eqref{e:palawan-mai2008-} and \eqref{e:2012-04-01a}, 
\begin{equation}
\label{e:2012-04-08g}
{\boldsymbol A}\;\;\text{is maximally monotone},
\end{equation}
${\boldsymbol D}$ is ($\min_{1\leq i\leq m}\nu_i$)-cocoercive,
Lemma~\ref{l:maxmon45}\ref{l:maxmon45ii} implies that 
\begin{equation}
\label{e:T}
-T\;\text{is $\rho$-cocoercive}, 
\end{equation}
while $\|{\boldsymbol L}\|^2\leq\sum_{i=1}^m\|L_i\|^2$. 
Hence, we derive from \eqref{e:beta1} and 
Proposition~\ref{p:samurai11} that
\begin{equation}
\label{e:2012-04-08f}
{\boldsymbol B}={\boldsymbol D}-{\boldsymbol L}T{\boldsymbol L}^*
\;\;\text{is}\;\;\beta\text{-cocoercive}.
\end{equation}
Moreover, it follows from \eqref{e:2012-04-08d}, 
\eqref{e:2012-04-01b}, and \eqref{e:palawan-mai2008-} that
\begin{equation}
\label{e:palawan2012-04-09a}
\sup_{n\in\NN}\|{\boldsymbol U}_n\|=\mu\quad\text{and}\quad
(\forall n\in\NN)\quad (1+\eta_n){\boldsymbol U}_{n+1}\succcurlyeq
{\boldsymbol U}_n\in\BP_\alpha(\GGG).
\end{equation}
Now set
\begin{equation}
\label{e:2012-04-08h}
(\forall n\in\NN)\quad 
\begin{cases}
{\boldsymbol a}_n=\big(b_{i,n}\big)_{1\leq i\leq m}\\
{\boldsymbol b}_n=
\big(d_{i,n}-L_ia_n\big)_{1\leq i\leq m}\\
{\boldsymbol v}_n=\big(v_{i,n}\big)_{1\leq i\leq m}\\
{\boldsymbol w}_n=\big(w_{i,n}\big)_{1\leq i\leq m}.
\end{cases}
\end{equation}
Then $\sum_{n\in\NN}|||{\boldsymbol a}_n|||<\pinf$,
$\sum_{n\in\NN}|||{\boldsymbol b}_n|||<\pinf$, and
\eqref{e:2012-04-08a} can be rewritten as 
\begin{equation}
\label{e:2012-04-08c}
(\forall n\in\NN)\quad
\begin{array}{l}
\left\lfloor
\begin{array}{l}
{\boldsymbol w}_n={\boldsymbol v}_n-\gamma_n {\boldsymbol U}_n
({\boldsymbol B}{\boldsymbol v}_n+{\boldsymbol b}_n)\\[1mm]
{\boldsymbol v}_{n+1}={\boldsymbol v}_n+\lambda_n
\big(J_{\gamma_n{\boldsymbol U}_n{\boldsymbol A}}\,
({\boldsymbol w}_n)+{\boldsymbol a}_n-{\boldsymbol v}_n\big).
\end{array}
\right.\\[2mm]
\end{array}
\end{equation}
Furthermore, the dual problem \eqref{e:fdual} is equivalent to
\begin{equation}
\label{e:2012-04-08b}
\text{find}\quad\overline{{\boldsymbol v}}\in\GGG\quad
\text{such that}\quad
{\boldsymbol 0}\in{\boldsymbol A}\overline{{\boldsymbol v}}+
{\boldsymbol B}\overline{{\boldsymbol v}}
\end{equation}
which, in view of \eqref{e:2012-04-08g}, \eqref{e:2012-04-08f}, 
and Proposition~\ref{p:2012-03-20}\ref{p:2012-03-20ii}, 
can be solved using \eqref{e:2012-04-08c}. Altogether,
the claims follow from Theorem~\ref{t:1}\ref{t:1i} and
Proposition~\ref{p:2012-03-20}\ref{p:2012-03-20iii}. 

\ref{c:2ii}: Set $(\forall n\in\NN)$ $z_n=x_n-a_n$. It follows 
from \ref{c:2i}, \eqref{e:2012-04-08a} and \eqref{e:2012-04-01a} 
that
\begin{equation}
\label{e:2012-04-09b}
\overline{x}=T({\boldsymbol L}^*{\overline{\boldsymbol v}})
\quad\text{and}\quad
(\forall n\in\NN)\quad z_n=T({\boldsymbol L}^*
{\boldsymbol v}_n).
\end{equation}
In turn, we deduce from \eqref{e:T}, \ref{c:2i}, 
\eqref{e:2012-04-08f}, and the monotonicity of 
${\boldsymbol D}$ that
\begin{align}
\label{e:2012-04-09c}
\rho\|z_n-\overline{x}\|^2
&=\rho\|T({\boldsymbol L}^*{\boldsymbol v}_n)-
T({\boldsymbol L}^*{\overline{\boldsymbol v}})\|^2
\nonumber\\
&\leq\scal{{\boldsymbol L}^*({\boldsymbol v}_n-
{\overline{\boldsymbol v}})}{T({\boldsymbol L}^*
{\overline{\boldsymbol v}})-T({\boldsymbol L}^*{\boldsymbol v}_n)}
\nonumber\\
&\leq\pscal{{\boldsymbol v}_n-{\overline{\boldsymbol v}}}
{{\boldsymbol L}T({\boldsymbol L}^*{\overline{\boldsymbol v}})
-{\boldsymbol L}T({\boldsymbol L}^*{\boldsymbol v}_n)}
\nonumber\\
&\leq
\pscal{{\boldsymbol v}_n-{\overline{\boldsymbol v}}}
{{\boldsymbol D}{\boldsymbol v}_n-
{\boldsymbol D}\overline{\boldsymbol v}}-
\pscal{{\boldsymbol v}_n-{\overline{\boldsymbol v}}}
{{\boldsymbol L}T({\boldsymbol L}^*{\boldsymbol v}_n)
-{\boldsymbol L}T({\boldsymbol L}^*{\overline{\boldsymbol v}})}
\nonumber\\
&=\pscal{{\boldsymbol v}_n-{\overline{\boldsymbol v}}}
{{\boldsymbol B}{\boldsymbol v}_n-
{\boldsymbol B}{\overline{\boldsymbol v}}}
\nonumber\\
&\leq\delta|||{\boldsymbol B}{\boldsymbol v}_n-
{\boldsymbol B}{\overline{\boldsymbol v}}|||,
\end{align}
where $\delta=\sup_{n\in\NN}|||{\boldsymbol v}_n-
{\overline{\boldsymbol v}}|||<\pinf$ by \ref{c:2i}.
Therefore, it follows from \eqref{e:2012-04-08c} and
Theorem~\ref{t:1}\ref{t:1ii} that $\|z_n-\overline{x}\|\to 0$.
Since $a_n\to 0$, we conclude that $x_n\to\overline{x}$.
\end{proof}

\begin{remark} 
Here are some observations on Corollary~\ref{c:2}. 
\begin{enumerate}
\item
At iteration $n$, the vectors $a_n$, $b_{i,n}$, and $d_{i,n}$
model errors in the implementation of the nonlinear operators.
Note also that, thanks to 
Example~\ref{ex:resolvent}\ref{ex:resolventi}, the computation of 
$v_{i,n+1}$ in \eqref{e:2012-04-08a} can be implemented using 
$J_{\gamma^{-1}_nU^{-1}_{i,n}B_i}$ rather than 
$J_{\gamma_nU_{i,n}B_i^{-1}}$.
\item 
Corollary~\ref{c:2} provides a general algorithm for solving 
strongly monotone composite inclusions which is new even in 
the fixed standard metric case, i.e.,
$(\forall i\in\{1,\ldots,m\})(\forall n\in\NN)$ $U_{i,n}=\Id$.
\end{enumerate}
\end{remark}

The following example describes an application of 
Corollary~\ref{c:2} to strongly convex minimization problems
which extends the primal-dual formulation 
\eqref{e:prob1}--\eqref{e:prob2} of \cite{Svva10} and 
solves it with a variable metric scheme. 
It also extends the framework of
\cite{Jmaa11}, where $f=0$ and $(\forall i\in\{1,\ldots,m\})$
$\ell_i=\iota_{\{0\}}$ and $(\forall n\in\NN)$ $U_{i,n}=\Id$.

\begin{example}
\label{ex:2012-04-11}
Let $z\in\HH$, let $f\in\Gamma_0(\HH)$, let $\alpha\in\RPP$, let 
$(\eta_n)_{n\in\NN}\in\ell_+^1(\NN)$, let $(a_n)_{n\in\NN}$ be 
an absolutely summable sequence in $\HH$, and let $m$ be a 
strictly positive integer. For every $i\in\{1,\ldots,m\}$, let 
$r_i\in\GG_i$, let $g_i\in\Gamma_0(\GG_i)$, let $\nu_i\in\RPP$,
let $\ell_i\in\Gamma_0(\GG_i)$ be $\nu_i$-strongly convex,
let $v_{i,0}\in\GG_i$, let $(b_{i,n})_{n\in\NN}$ and 
$(d_{i,n})_{n\in\NN}$ be absolutely 
summable sequences in $\GG_i$, let $(U_{i,n})_{n\in\NN}$ be 
a sequence in $\BP_{\alpha}(\GG_i)$, and suppose that 
$0\neq L_i\in\BL(\HH,\GG_i)$. Furthermore, suppose that
(see Proposition~\ref{p:2012-03-20}\ref{p:2012-03-20ivb} for
special cases)
\begin{equation}
\label{e:2012-04-18d}
z\in\ran\bigg(\partial f+\sum_{i=1}^mL_i^*(\partial g_i\infconv
\partial\ell_i)(L_i\cdot-r_i)+\Id\bigg).
\end{equation}
The primal problem is to
\begin{equation}
\label{e:primal}
\minimize{x\in\HH}{f(x)+\sum_{i=1}^m\,(g_i\infconv\ell_i)
(L_ix-r_i)+\frac{1}{2}\|x-z\|^2}, 
\end{equation}
and the dual problem is to
\begin{equation}
\label{e:dual}
\minimize{v_1\in\GG_1,\ldots,v_m\in\GG_m}{
\widetilde{f^*}\bigg(z-\sum_{i=1}^mL_i^*v_i\bigg)
+\sum_{i=1}^m\big(g_i^*(v_i)+\ell_i^*(v_i)+\scal{v_i}{r_i}\big)}.
\end{equation}
Suppose that \eqref{e:2012-04-08d} holds, let 
$\varepsilon\in\left]0,\min\{1,2\beta/(\mu+1)\}\right]$,
let $(\lambda_n)_{n\in\NN}$ be a sequence in 
$\left[\varepsilon,1\right]$, and let $(\gamma_n)_{n\in\NN}$ be a 
sequence in $[\varepsilon,(2\beta-\varepsilon)/\mu]$ where 
$\beta$ is defined in \eqref{e:beta1} and $\mu$ in 
\eqref{e:2012-04-08d}. Set
\begin{equation}
\label{e:2012-04-11a}
(\forall n\in\NN)\quad
\left\lfloor
\begin{array}{l}
s_n=z-\sum_{i=1}^mL_i^*v_{i,n}\\
x_n=\prox_fs_n+a_n\\
\operatorname{For}\;i=1,\ldots,m\\
\left\lfloor
\begin{array}{l}
w_{i,n}=v_{i,n}+\gamma_nU_{i,n}\big(L_ix_n-r_i-
\nabla\ell_i^*(v_{i,n})-d_{i,n}\big)\\[2mm]
v_{i,n+1}=v_{i,n}+\lambda_n
\Big(\prox^{U_{i,n}^{-1}}_{\gamma_ng_i^{*}}
w_{i,n}+b_{i,n}-v_{i,n}\Big).\\[1mm] 
\end{array}
\right.\\[2mm]
\end{array}
\right.
\end{equation}
Then \eqref{e:primal} admits a unique solution $\overline{x}$
and the following hold for some solution 
$(\overline{v_1},\ldots,\overline{v_m})$ to \eqref{e:dual}.
\begin{enumerate}
\item
\label{ex:2012-04-11i}
$(\forall i\in\{1,\ldots,m\})$ $v_{i,n}\weakly\overline{v_i}$.
In addition, 
$\overline{x}=\prox_f(z-\sum_{i=1}^mL_i^*\overline{v_i})$.
\item
\label{ex:2012-04-11ii} 
$x_n\to\overline{x}$.
\end{enumerate}
\end{example}
\begin{proof}
Set $A=\partial f$ and, for every $i\in\{1,\ldots,m\}$,
$B_i=\partial g_i$ and $D_i=\partial\ell_i$. In this setting,
it follows from the analysis of \cite[Section~4]{Svva12} that
\eqref{e:primal}--\eqref{e:dual} is a special 
case of Problem~\ref{prob:5.1} and, using \eqref{e:prox3},
that \eqref{e:2012-04-11a} is a special case of 
\eqref{e:2012-04-08a}. Altogether, the claims follow from 
Corollary~\ref{c:2}.
\end{proof}

We conclude this section with an application to a composite best 
approximation problem. 

\begin{example}
\label{ex:2012-04-18}
Let $z\in\HH$, let $C$ be a closed convex subset of $\HH$, let 
$\alpha\in\RPP$, let $(\eta_n)_{n\in\NN}\in\ell_+^1(\NN)$, let 
$(a_n)_{n\in\NN}$ be an absolutely summable sequence in $\HH$, 
and let $m$ be a strictly positive integer. For every 
$i\in\{1,\ldots,m\}$, let $r_i\in\GG_i$, let $D_i$ be a closed 
convex subset of $\GG_i$, let $v_{i,0}\in\GG_i$, let 
$(b_{i,n})_{n\in\NN}$ be an absolutely 
summable sequence in $\GG_i$, let $(U_{i,n})_{n\in\NN}$ be 
a sequence in $\BP_{\alpha}(\GG_i)$, and suppose that 
$0\neq L_i\in\BL(\HH,\GG_i)$. The problem is to
\begin{equation}
\label{e:bprimal}
\minimize{\substack{x\in C\\ L_1x\in r_1+D_1\\ \vdots\\
L_mx\in r_m+D_m}}{\|x-z\|}.
\end{equation}
Suppose that \eqref{e:2012-04-08d} holds,
that $(\max_{1\leq i\leq m}\sup_{n\in\NN}\|U_{i,n}\|)
\sum_{i=1}^m\|L_i\|^2<2$, and that
\begin{equation}
\label{e:2012-04-18b}
(r_1,\ldots,r_m)\in\sri\menge{(L_ix-y_i)_{1\leq i\leq m}}
{x\in C\;\text{and}\;(\forall i\in\{1,\ldots,m\})\;\:y_i\in D_i}.
\end{equation}
Set
\begin{equation}
\label{e:2012-04-18a}
(\forall n\in\NN)\quad
\left\lfloor
\begin{array}{l}
s_n=z-\sum_{i=1}^mL_i^*v_{i,n}\\
x_n=P_Cs_n+a_n\\
\operatorname{For}\;i=1,\ldots,m\\
\left\lfloor
\begin{array}{l}
w_{i,n}=v_{i,n}+U_{i,n}\big(L_ix_n-r_i\big)\\[2mm]
v_{i,n+1}=w_{i,n}-U_{i,n}\Big(P_{D_i}^{U_{i,n}}
\big(U_{i,n}^{-1}w_{i,n}\big)+b_{i,n}\Big).\\[1mm] 
\end{array}
\right.\\[2mm]
\end{array}
\right.
\end{equation}
Then $(x_n)_{n\in\NN}$ converges strongly to the 
unique solution $\overline{x}$ to \eqref{e:bprimal}. 
\end{example}
\begin{proof}
Set $f=\iota_C$ and $(\forall i\in\{1,\ldots,m\})$ 
$g_i=\iota_{D_i}$, $\ell_i=\iota_{\{0\}}$, and 
$(\forall n\in\NN)$ $\gamma_n=\lambda_n=1$ and
$d_{i,n}=0$. Then \eqref{e:2012-04-18b} and 
Proposition~\ref{p:2012-03-20}\ref{p:heidelberg2011-07-06i} 
imply that 
\eqref{e:2012-04-18d} is satisfied. Moreover, in view of
Example~\ref{ex:resolvent}\ref{ex:resolventiii}, 
\eqref{e:2012-04-18a} is a special case of \eqref{e:2012-04-11a}.
Hence, the claim follows from 
Example~\ref{ex:2012-04-11}\ref{ex:2012-04-11ii}.
\end{proof}

\section{Inclusions involving cocoercive operators}
\label{sec:6}

We revisit a primal-dual problem investigated first in
\cite{Svva12}, and then in \cite{Bang12} with the scenario 
described below.

\begin{problem}
\label{prob:8}
Let $z\in\HH$, let $A\colon\HH\to 2^{\HH}$ be maximally monotone, 
let $\mu\in\RPP$, let $C\colon\HH\to\HH$ be $\mu$-cocoercive, 
and let $m$ be a strictly positive integer.
For every $i\in\{1,\ldots, m\}$, let $r_i\in\GG_i$, 
let $B_i\colon \GG_i\to2^{\GG_i}$ be maximally monotone, 
let $\nu_i\in\RPP$,
let $D_i\colon \GG_i \to 2^{\GG_i}$ be maximally monotone and
$\nu_i$-strongly monotone,
and suppose that $0\neq L_i\in \BL(\HH,\GG_i)$. 
The problem is to solve the primal inclusion
\begin{equation}
\label{e:primal12:41}
\text{find}\quad\overline{x}\in\HH
\quad\text{such that}\quad
z\in A\overline{x}+\sum_{i=1}^mL^{*}_i
\big((B_i\infconv D_i)(L_i\overline{x}-r_i)\big)+C\overline{x},
\end{equation}
together with the dual inclusion
\begin{multline}
\label{e:dual12:41}
\text{find}\;\;\overline{v_1}\in\GG_1,\:\ldots,\:
\overline{v_m}\in\GG_m\;\:\text{such that}\\
(\exi x\in\HH)\quad
\begin{cases}
z-\sum_{i=1}^m L_{i}^*\overline{v}_i \in Ax+Cx\\
(\forall i\in\{1,\ldots,m\})\;\overline{v}_i\in 
(B_i\infconv D_i)(L_ix-r_i).
\end{cases}
\end{multline}
\end{problem}

\begin{corollary}
\label{c:aicm}
In Problem~\ref{prob:8}, suppose that 
\begin{equation}
\label{e:1rang820}
z\in\ran
\bigg(A+\sum_{i=1}^mL^{*}_i
\big((B_i\infconv D_i)(L_i\cdot-r_i)\big)+C\bigg),
\end{equation}
and set
\begin{equation}
\beta=\min\{\mu,\nu_1,\ldots,\nu_m\}.
\end{equation}
Let $\varepsilon\in\left]0,\min\{1,\beta\}\right[$,
let $\alpha\in\RPP$, let $(\lambda_n)_{n\in\NN}$ be a sequence in 
$\left[\varepsilon,1\right]$, let $x_0\in\HH$, let 
$(a_n)_{n\in\NN}$ and $(c_n)_{n\in\NN}$ be absolutely summable 
sequences in $\HH$, and let $(U_n)_{n\in\NN}$ be a sequence 
in $\BP_{\alpha}(\HH)$ such that 
$(\forall n\in\NN)\; U_{n+1}\succcurlyeq U_n$.
For every $i\in\{1,\ldots,m\}$, let
$v_{i,0}\in\GG_i$, and let $(b_{i,n})_{n\in\NN}$ 
and $(d_{i,n})_{n\in\NN}$ be  
absolutely summable sequences in $\GG_i$, and let
$(U_{i,n})_{n\in\NN}$ be a sequence in 
$\BP_{\alpha}(\GG_i)$ such that
$(\forall n\in\NN)$ $U_{i,n+1}\succcurlyeq U_{i,n}$.
For every $n\in\NN$, set
\begin{equation}
\label{e:121}
\delta_n=
\Bigg(\sqrt{\sum_{i= 1}^m\|\sqrt{U_{i,n}} 
L_i\sqrt{U_n}\|^2}\Bigg)^{-1}-1,
\end{equation}
and suppose that
\begin{equation}
\label{e:2f9h79p}
\zeta_n=\frac{\delta_n}
{(1+\delta_n)\max\{\|U_n\|,\|U_{1,n}\|,\ldots,\|U_{m,n}\|\}}
\geq\frac{1}{2\beta-\varepsilon}.
\end{equation}
Set
\begin{equation}
\label{e:cocoeqal8:20}
(\forall n\in\NN)\quad 
\begin{array}{l}
\left\lfloor
\begin{array}{l}
p_n=J_{U_nA}\Big(x_n-U_n
\big(\sum_{i=1}^mL_{i}^*v_{i,n}+C x_n+c_n-z\big)\Big)+a_n\\
y_n=2p_n-x_n\\
x_{n+1}=x_n+\lambda_n(p_n-x_n)\\
\operatorname{For}\;i=1,\ldots, m\\
\left\lfloor
\begin{array}{l}
q_{i,n}=J_{U_{i,n}B_{i}^{-1}}
\Big(v_{i,n}+U_{i,n}\big(L_iy_n- D_{i}^{-1}v_{i,n} 
-d_{i,n}-r_i\big)\Big)+b_{i,n}\\
v_{i,n+1}=v_{i,n}+\lambda_n(q_{i,n}-v_{i,n}).\\
\end{array}
\right.\\[2mm]
\end{array}
\right.\\[2mm]
\end{array}
\end{equation}
Then the following hold for some solution $\overline{x}$ to 
\eqref{e:primal12:41} and some solution 
$(\overline{v_1},\ldots,\overline{v_m})$ to \eqref{e:dual12:41}.
\begin{enumerate}
\item
\label{c:aicmi}
$x_n\weakly\overline{x}$.
\item
\label{c:aicmii}
$(\forall i\in\{1,\ldots,m\})$ $v_{i,n}\weakly\overline{v_i}$. 
\item
\label{c:aicmiii}
Suppose that $C$ is demiregular at $\overline{x}$.
Then $x_n\to\overline{x}$.
\item
\label{c:aicmiv}
Suppose that, for some $j\in\{1,\ldots,m\}$, $D_j^{-1}$ is 
demiregular at $\overline{v_j}$. Then 
$v_{j,n}\to\overline{v_j}$.
\end{enumerate}
\end{corollary}
\begin{proof} 
Define $\GGG$ as in Notation \ref{n:palawan-mai2008-} 
and set $\KKK=\HH\oplus\GGG$. We denote the scalar product and the
norm of $\KKK$ by $\psscal{\cdot}{\cdot}$ and $||||\cdot||||$, 
respectively. As shown in \cite{Svva12,Bang12}, the operators
\begin{equation}
\begin{cases}
\label{e:maximal1}
{\boldsymbol A}\hskip -3mm&\colon\KKK\to 2^{\KKK}\colon
(x,v_1,\ldots,v_m)\mapsto (\sum_{i=1}^mL_{i}^*v_i-z+Ax)
\times(r_1-L_1x+B_{1}^{-1}v_1)\times\ldots\times\\
&\hskip 92mm (r_m-L_mx+B^{-1}_{m}v_m)\\
{\boldsymbol B}\hskip -3mm&\colon\KKK\to\KKK\colon
(x,v_1,\ldots,v_m)\mapsto 
\big(Cx,D^{-1}_1v_1,\ldots,D^{-1}_mv_m\big)\\[2mm]
{\boldsymbol S}\hskip -3mm&\colon\KKK\to \KKK\colon
(x,v_1,\ldots,v_m)\mapsto
\bigg(\sum_{i=1}^mL_{i}^*v_i,-L_1x,\ldots,-L_mx\bigg)
\end{cases}
\end{equation}
are maximally monotone and, moreover, ${\boldsymbol B}$ is 
$\beta$-cocoercive \cite[Eq.~(3.12)]{Bang12}. Furthermore, as shown
in \cite[Section~3]{Svva12}, under condition \eqref{e:1rang820},
$\zer({\boldsymbol A}+{\boldsymbol B})\neq\emp$ and
\begin{equation}
\label{e:khoqua}
(\overline{x},\overline{{\boldsymbol v}}) 
\in\zer({\boldsymbol A}+{\boldsymbol B})
\quad\Rightarrow\quad
\overline{x}\;\:\text{solves \eqref{e:primal12:41} and}\;\;
\overline{{\boldsymbol v}}\;\;\text{solves \eqref{e:dual12:41}}.
\end{equation}
Next, for every $n\in\NN$, define
\begin{equation}
\begin{cases}
\label{e:uvt}
{\boldsymbol U}_n\colon\KKK\to\KKK\colon
(x,v_1,\ldots, v_m)
\mapsto\Big(U_nx,U_{1,n}v_1,\ldots,U_{m,n}v_m\Big)\\[2mm]
{\boldsymbol V}_n\colon\KKK\to \KKK\colon
(x,v_1,\ldots,v_m)\mapsto 
\bigg(U_n^{-1}x-\sum_{i=1}^m L^{*}_iv_i,\big(-L_ix+ 
U_{i,n}^{-1}v_i\big)_{1\leqslant i\leqslant m}\bigg)\\
{\boldsymbol T}_n\colon\HH\to\GGG\colon 
x\mapsto\Big(\sqrt{U_{1,n}}L_1x,\ldots,\sqrt{U_{m,n}}L_mx\Big).
\end{cases}
\end{equation}
It follows from our assumptions and 
Lemma~\ref{l:kjMMXII}\ref{l:kjMMXII-iii} that 
\begin{equation}
\label{e:8/11a}
(\forall n\in\NN)\quad
{\boldsymbol U}_{n+1}\succcurlyeq {\boldsymbol U}_n\in 
\BP_{\alpha}(\KKK)\quad\text{and}\quad
||{\boldsymbol U}_n^{-1}||\leq\frac{1}{\alpha}.
\end{equation}
Moreover, for every $n\in\NN$, ${\boldsymbol V}_n\in\SL(\KKK)$ 
since ${\boldsymbol U}_n\in\SL(\KKK)$. In addition, \eqref{e:uvt} 
and \eqref{e:8/11a} yield
\begin{equation}
\label{e:baba2}
\big(\forall n\in\NN\big)\quad\|{\boldsymbol V}_n\| 
\leq\|{\boldsymbol U}^{-1}_n\|+\|{\boldsymbol S}\|\leq\rho,
\quad\text{where}\quad
\rho=\frac{1}{\alpha}+\sqrt{\sum_{i=1}^m\|L_i\|^2}.
\end{equation}
On the other hand,
\begin{align}
\label{e:muaxuan1aa}
(\forall n\in\NN)(\forall x\in\HH)\quad 
|||{\boldsymbol T}_n x|||^2
&=\sum_{i=1}^m\big\|\sqrt{U_{i,n}}
L_i\sqrt{U_n}\sqrt{U_n}^{\:-1}x\big\|^2\nonumber\\
&\leq\|x\|_{U^{-1}_n}^2\sum_{i=1}^m\big\|\sqrt{U_{i,n}}L_i
\sqrt{U_n}\big\|^2\nonumber\\
&=\beta_n\|x\|_{U^{-1}_n}^2,
\end{align}
where $(\forall n\in\NN)$
$\beta_n=\sum_{i=1}^m\big\|\sqrt{U_{i,n}}L_i\sqrt{U_n}\big\|^2$.
Hence, \eqref{e:121} yields
\begin{equation}
\label{e:q407}
(\forall n\in\NN)\quad(1+\delta_n)\beta_n=\frac{1}{1+\delta_n}.
\end{equation}
Therefore, for every $n\in\NN$ and every 
${\boldsymbol x}=(x,v_1,\ldots,v_m)\in\KKK$, using 
\eqref{e:uvt}, \eqref{e:muaxuan1aa}, \eqref{e:q407},
Lemma~\ref{l:kjMMXII}\ref{l:kjMMXII-ii}, and \eqref{e:2f9h79p}, 
we obtain
\begin{align}
\psscal{{\boldsymbol x}}{{\boldsymbol V}_n {\boldsymbol x}}
&=\scal{x}{U_n^{-1}x}+ 
\sum_{i=1}^m\scal{v_i}{U^{-1}_{i,n}v_{i}}
-2\sum_{i=1}^m\scal{L_ix}{v_i}\nonumber\\
&=\|x\|_{U^{-1}_n}^{2}+\sum_{i=1}^m\|v_i\|_{U^{-1}_{i,n}}^{2}
-2\sum_{i=1}^m\scal{\sqrt{U_{i,n}}L_ix}{\sqrt{U_{i,n}}^{\:-1}v_i}
\nonumber\\
&=\|x\|_{U^{-1}_n}^{2}+\sum_{i=1}^m\|v_i\|_{U^{-1}_{i,n}}^{2}
\nonumber\\ 
&\quad\;-2\pscal{\sqrt{(1+\delta_n)\beta_n}^{\,-1}{\boldsymbol T}_nx}
{\sqrt{(1+\delta_n)\beta_n}\big(\sqrt{U_{1,n}}^{\:-1} v_1,\ldots,
\sqrt{U_{m,n}}^{\:-1} v_m \big)}\nonumber\\
&\geq\|x\|_{U^{-1}_n}^{2}+ 
\sum_{i=1}^m\|v_i\|_{U^{\:-1}_{i,n}}^{2}
-\bigg(\frac{|||{\boldsymbol T}_n x|||^2}{(1+\delta_n)\beta_n}
+(1+\delta_n)\beta_n
\sum_{i=1}^m\|v_i\|^{2}_{U^{-1}_{i,n}}\bigg)\nonumber\\
&\geq\|x\|_{U^{-1}_n}^{2}+ 
\sum_{i=1}^m\|v_i\|_{U^{\:-1}_{i,n}}^{2}
-\bigg(\frac{\|x\|^2_{U^{-1}_n}}{(1+\delta_n)}
+(1+\delta_n)\beta_n
\sum_{i=1}^m\|v_i\|^{2}_{U^{-1}_{i,n}}\bigg)\nonumber\\
&=\frac{\delta_n}{1+\delta_n}\Big(\|x\|_{U^{-1}_n}^{2}+ 
\sum_{i=1}^m\|v_i\|_{U^{\:-1}_{i,n}}^{2}\big)\nonumber\\
&\geq\frac{\delta_n}{1+\delta_n}\bigg(\|U_n\|^{-1}\|x\|^2 
+\sum_{i=1}^m\|U_{i,n}\|^{-1}\|v_i\|^{2} \bigg)\nonumber\\
&\geq\zeta_n||||{\boldsymbol x}||||^{2}\label{e:ancan1}.
\end{align}
In turn, it follows from Lemma~\ref{l:kjMMXII}\ref{l:kjMMXII-iii} 
and \eqref{e:2f9h79p} that 
\begin{equation}
\label{e:conkhi}
(\forall n\in\NN)\quad\|{\boldsymbol V}_n^{-1}\|\leq
\frac{1}{\zeta_n}\leq 2\beta-\varepsilon.
\end{equation}
Moreover, by Lemma~\ref{l:kjMMXII}\ref{l:kjMMXII-i}, 
$(\forall n\in\NN)$ 
$({\boldsymbol U}_{n+1}\succcurlyeq {\boldsymbol U}_n$
$\Rightarrow$
${\boldsymbol U}_n^{-1}\succcurlyeq {\boldsymbol U}^{-1}_{n+1}$
$\Rightarrow$
${\boldsymbol V}_n\succcurlyeq {\boldsymbol V}_{n+1}$
$\Rightarrow$
${\boldsymbol V}_{n+1}^{-1}\succcurlyeq {\boldsymbol V}_n^{-1})$.
Furthermore, we derive from Lemma~\ref{l:kjMMXII}\ref{l:kjMMXII-ii}
and \eqref{e:baba2} that
\begin{equation}
\label{e:barca}
(\forall{\boldsymbol x}\in\KKK)\quad\psscal{{\boldsymbol V}_n^{-1}
{\boldsymbol x}}{{\boldsymbol x}}\geq
\|{\boldsymbol V}_n\|^{-1}||||{\boldsymbol x}||||^{2}
\geq\frac{1}{\rho}||||{\boldsymbol x}||||^2.
\end{equation}
Altogether, 
\begin{equation}
\label{e:manu}
\sup_{n\in\NN}\|{\boldsymbol V}_{n}^{-1}\|\leq 2\beta-\varepsilon\\
\quad\text{and}\quad
(\forall n\in\NN)\quad 
{\boldsymbol V}_{n+1}^{-1}\succcurlyeq{\boldsymbol V}_n^{-1}
\in{\EuScript P}_{1/\rho}(\KKK).
\end{equation}
Now set, for every $n\in\NN$, 
\begin{equation}
\label{e:2012-05-01}
\begin{cases}
{\boldsymbol x}_n=(x_n,v_{1,n},\ldots,v_{m,n})\\
{\boldsymbol y}_n=(p_n,q_{1,n},\ldots,q_{m,n})\\
{\boldsymbol a}_n=(a_n,b_{1,n},\ldots,b_{m,n})\\
{\boldsymbol c}_n=(c_n,d_{1,n},\ldots,d_{m,n})\\
{\boldsymbol d}_n=(U_n^{-1}a_n,U_{1,n}^{-1}b_{1,n},\ldots, 
U_{m,n}^{-1}b_{m,n})
\end{cases}
\quad\text{and}\quad
{\boldsymbol b}_n=({\boldsymbol S}
+{\boldsymbol V}_n){\boldsymbol a}_n+{\boldsymbol c}_n
-{\boldsymbol d}_n.
\end{equation}
Then 
$\sum_{n\in\NN}||||{\boldsymbol a}_n||||<\pinf$,
$\sum_{n\in\NN}||||{\boldsymbol c}_n||||<\pinf$,
and $\sum_{n\in\NN}|||| {\boldsymbol d}_n||||<\pinf$.
Therefore \eqref{e:baba2} implies that
$\sum_{n\in\NN}||||{\boldsymbol b}_n||||<\pinf$.
Furthermore, using the same arguments as in 
\cite[Eqs.~(3.22)--(3.35)]{Bang12}, we derive from 
\eqref{e:cocoeqal8:20} and \eqref{e:maximal1} that
\begin{align}
\label{e:fb9}
(\forall n\in\NN)\quad
{\boldsymbol x}_{n+1}
&={\boldsymbol x}_n+\lambda_n\Big(J_{{\boldsymbol V}_n^{-1}
{\boldsymbol A}}\big({\boldsymbol x}_n-
{\boldsymbol V}_{n}^{-1}({\boldsymbol B}{\boldsymbol x}_n+
{\boldsymbol b}_n)\big)+{\boldsymbol a}_n-
{\boldsymbol x}_n\Big).
\end{align}
We observe that \eqref{e:fb9} has the structure of the 
variable metric forward-backward splitting algorithm 
\eqref{e:forward}, where 
$(\forall n\in\NN)$ $\gamma_n=1$. Finally, \eqref{e:conkhi} and 
\eqref{e:manu} imply that all the conditions in Theorem~\ref{t:1} 
are satisfied. 

\ref{c:aicmi}\&\ref{c:aicmii}:
Theorem~\ref{t:1}\ref{t:1i} asserts that there exists
\begin{equation}
\label{e:2012-05-10}
\overline{{\boldsymbol x}}=(\overline{x},\overline{v_1},\ldots,
\overline{v_m})\in\zer({\boldsymbol A}+{\boldsymbol B}) 
\end{equation}
such that ${\boldsymbol x}_n\weakly\overline{{\boldsymbol x}}$.
In view of \eqref{e:khoqua}, the assertions are proved.

\ref{c:aicmiii}\&\ref{c:aicmiv}:
It follows from Theorem~\ref{t:1}\ref{t:1ii} that 
${\boldsymbol B}{\boldsymbol x}_n\to{\boldsymbol B}
\overline{{\boldsymbol x}}$. Hence, \eqref{e:maximal1},
\eqref{e:2012-05-01}, and \eqref{e:2012-05-10} yield 
\begin{equation}
\label{e:fin1}
Cx_n\to C\overline{x} \quad\text{and}\quad 
\big(\forall i\in\{1,\ldots,m\}\big)\quad
D_{i}^{-1}v_{i,n}\to D_{i}^{-1}\overline{v_{i}}. 
\end{equation}
Hence the results follow from \ref{c:aicmi}\&\ref{c:aicmii}
and Definition~\ref{d:demir}.
\end{proof}

\begin{remark}
\label{r:2012-05-05}
In the case when $C=\rho\Id$ for some $\rho\in\RPP$, 
Problem~\ref{prob:8} reduces to Problem~\ref{prob:5.1}. 
However, the algorithm obtained in 
Corollary~\ref{p:2012-03-20} is quite different from that of
Corollary~\ref{c:aicm}. Indeed, the former was obtained by 
applying the forward-backward algorithm \eqref{e:forward}
to the dual inclusion, which was made possible by the strong 
monotonicity of the primal problem. By contrast, the latter 
relies on an application of \eqref{e:forward} in a primal-dual 
product space.
\end{remark}

\begin{example}
\label{ex:2012-30-04}
Let $z\in\HH$, let $f\in\Gamma_0(\HH)$, 
let $\mu\in\RPP$, let $h\colon\HH\to\RR$ be convex and 
differentiable with a $\mu^{-1}$-Lipschitzian gradient,
let $(a_n)_{n\in\NN}$ and $(c_n)_{n\in\NN}$ be absolutely summable 
sequences in $\HH$, let $\alpha\in\RPP$, let $m$ be a strictly 
positive integer, and let $(U_n)_{n\in\NN}$ be a sequence 
in $\BP_{\alpha}(\HH)$ such that 
$(\forall n\in\NN)$ $U_{n+1}\succcurlyeq U_n$.
For every $i\in\{1,\ldots,m\}$, let 
$r_i\in\GG_i$, let $g_i\in\Gamma_0(\GG_i)$, let $\nu_i\in\RPP$,
let $\ell_i\in\Gamma_0(\GG_i)$ be $\nu_i$-strongly convex,
let $v_{i,0}\in\GG_i$, let $(b_{i,n})_{n\in\NN}$ and 
$(d_{i,n})_{n\in\NN}$ be absolutely summable sequences in $\GG_i$, 
suppose that $0\neq L_i\in\BL(\HH,\GG_i)$, and let 
$(U_{i,n})_{n\in\NN}$ be a sequence in $\BP_{\alpha}(\GG_i)$ 
such that $(\forall n\in\NN)$ $U_{i,n+1}\succcurlyeq U_{i,n}$.
Furthermore, suppose that
\begin{equation}
\label{e:2012-30-04}
z\in\ran\bigg(\partial f+\sum_{i=1}^mL_i^*(\partial g_i\infconv
\partial\ell_i)(L_i\cdot-r_i)+\nabla h\bigg).
\end{equation}
The primal problem is to
\begin{equation}
\label{e:primalex6}
\underset{x\in\HH}{\text{minimize}} \;
f(x)+\sum_{i=1}^m(g_i\infconv\ell_i)(L_ix-r_i)+h(x)-\scal{x}{z},
\end{equation}
and the dual problem is to
\begin{equation}
\label{e:dualex6}
\underset{v_1\in\GG_1,\ldots,v_m\in\GG_m}{\text{minimize}} \;
(f^*\infconv h^*)\bigg(z-\sum_{i=1}^m L_{i}^*v_i\bigg)+
\sum_{i=1}^m \big(g^{*}_i(v_i)
+\ell^{*}_i(v_i)+\scal{v_i}{r_i}\big).
\end{equation}
Let $\beta=\min\{\mu,\nu_1,\ldots,\nu_m\}$,
let $\varepsilon\in\left]0,\min\{1,\beta\}\right[$,
let $(\lambda_n)_{n\in\NN}$ be a sequence in 
$\left[\varepsilon,1\right]$,
suppose that \eqref{e:2f9h79p} holds, and set 
\begin{equation}
\label{e:cocoeqal8:20coco}
(\forall n\in\NN)\quad 
\begin{array}{l}
\left\lfloor
\begin{array}{l}
p_n=\prox^{U^{-1}_n}_f\Big(x_n-U_n
\big(\sum_{i=1}^{m}L_{i}^*v_{i,n}+
\nabla h(x_n)+c_n-z\big)\Big)+a_n\\
y_n=2p_n-x_n\\
x_{n+1}=x_n+\lambda_n(p_n-x_n)\\
\operatorname{For}\;  i=1,\ldots, m\\
\left\lfloor
\begin{array}{l}
q_{i,n}=\prox^{U_{i,n}^{-1}}_{g_{i}^{*}}
\Big(v_{i,n}+U_{i,n}\big(L_iy_n- 
\nabla\ell_{i}^{*}(v_{i,n})-d_{i,n}-r_i\big)\Big)+b_{i,n}\\
v_{i,n+1}=v_{i,n}+\lambda_n(q_{i,n}-v_{i,n}).\\
\end{array}
\right.\\[2mm]
\end{array}
\right.\\[2mm]
\end{array}
\end{equation}
Then $(x_n)_{n\in\NN}$ converges weakly to a solution to 
\eqref{e:primalex6}, for every $i\in\{1,\ldots,m\}$ 
$(v_{i,n})_{n\in\NN}$ converges weakly to some
$\overline{v_i}\in\GG_i$, and 
$(\overline{v}_{1},\ldots,\overline{v}_{m})$ is 
a solution to \eqref{e:dualex6}.
\end{example}
\begin{proof}
Set $A=\partial f$, $C=\nabla h$, and 
$(\forall i\in\{1,\ldots,m\})$ 
$B_i=\partial g_i$ and $D_i=\partial \ell_i$. In this setting,
it follows from the analysis of \cite[Section~4]{Svva12} that
\eqref{e:primalex6}--\eqref{e:dualex6} is a special 
case of Problem~\ref{prob:8} and, using \eqref{e:prox3},
that \eqref{e:cocoeqal8:20coco} is a special case of 
\eqref{e:cocoeqal8:20}. Thus, the claims follow from 
Corollary~\ref{c:aicm}\ref{c:aicmi}\&\ref{c:aicmii}.
\end{proof}

\begin{remark} 
Suppose that, in Corollary~\ref{c:aicm} and 
Example~\ref{ex:2012-30-04}, there exist $\tau$ 
and $(\sigma_i)_{1\leq i\leq m}$ in $\RPP$ such that
$(\forall n\in\NN)$ $U_n=\tau\Id$ and 
$(\forall i\in\{1,\ldots,m\})$ $U_{i,n}=\sigma_i\Id$.
Then \eqref{e:cocoeqal8:20} and \eqref{e:cocoeqal8:20coco} 
reduce to the fixed metric methods appearing in 
\cite[Eq.~(3.3)]{Bang12} and \cite[Eq.~(4.5)]{Bang12}, 
respectively (see \cite{Bang12} for further connections with
existing work).
\end{remark}

\end{document}